\documentclass[a4paper,11pt,reqno]{amsart}
\usepackage{amsmath}
\usepackage{relsize}
\usepackage[nospace,noadjust]{cite}
\usepackage{color}
\usepackage[dvipsnames]{xcolor}
\usepackage[normalem]{ulem}

\usepackage{hyperref, enumitem}
\hypersetup{
  colorlinks   = true, %Colours links instead of ugly boxes
  urlcolor     = blue, %Colour for external hyperlinks
  linkcolor    = Purple, %Colour of internal links%
  citecolor   = Purple %Colour of citations
}
\usepackage{amsmath,amsthm,amssymb}
\usepackage{cleveref}

\usepackage[margin=2.5cm]{geometry}
\usepackage{caption} 

\usepackage{tikz-cd}
\usepackage{tikz-3dplot}
\usetikzlibrary{calc,fit,matrix,arrows,automata,positioning, shadows}
\usetikzlibrary{decorations.pathreplacing,angles,quotes}
\usepackage{stmaryrd}

\usetikzlibrary{fit}
\tikzset{%
  highlight/.style={rectangle,rounded corners,fill=red!15,draw,
    fill opacity=0.5,thick,inner sep=0pt}
}

\usepackage{array}
\newcolumntype{x}[1]{>{\centering\arraybackslash\hspace{0pt}}p{#1}}

\newcommand{\ras}[1]{\renewcommand{\arraystretch}{#1}}

\theoremstyle{definition}
\newtheorem{theorem}{Theorem}[section]
\newtheorem{definition}[theorem]{{{Definition}}}
\newtheorem{example}[theorem]{{{Example}}}
\newtheorem{notation}[theorem]{{{Notation}}}
\newtheorem{remark}[theorem]{{{Remark}}}

\newtheorem{corollary}[theorem]{{{Corollary}}}%[theorem]
\newtheorem{proposition}[theorem]{{{Proposition}}}
\newtheorem{lemma}[theorem]{{{Lemma}}}

\newcommand{\C}{\mathcal C}
\newcommand{\mC}{\mathcal C}

\newcommand{\mM}{\mathcal{M}}
\newcommand{\mL}{\mathcal{L}}
\newcommand{\mU}{\mathcal{U}}

\newcommand{\mS}{\mathcal{S}}

\newcommand{\mN}{\mathcal{N}}

\newcommand{\mA}{\mathcal{A}}
\newcommand{\F}{\mathbb F}

\newcommand{\la}{\langle}
\newcommand{\ra}{\rangle}
\newcommand{\mB}{\mathcal{B}}
\newcommand{\mP}{\mathcal{P}}

\newcommand{\PP}{\mathbb{P}}
\newcommand{\Proj}{\operatorname{Proj}}

\newcommand{\Fq}{\F_q}

\DeclareMathOperator{\GL}{GL}
\DeclareMathOperator{\supp}{supp}

\DeclareMathOperator{\rk}{rk}

\DeclareMathOperator{\dd}{d}

\DeclareMathOperator{\w}{w}

\DeclareMathOperator{\rowsp}{rowsp}
\DeclareMathOperator{\colsp}{colsp}
\DeclareMathOperator{\sss}{ss}

\usepackage{scrextend}
\allowdisplaybreaks

\title{Representability of $q$-matroids via rank-metric codes}
\usepackage[foot]{amsaddr}
\author[G. N. Alfarano]{Gianira N. Alfarano$^{1}$}
\address{$^1$\textnormal{Université de Rennes, IRMAR, Rennes, France.}}
 \email{gianira-nicoletta.alfarano@univ-rennes.fr}
 \author[S. Degen]{Sebastian Degen$^2$}
\address{$^2$\textnormal{Universit\"{a}t Bielefeld, Fakult\"{a}t f\"{u}r Mathematik, Bielefeld, Germany.}}
 \email{sdegen@math.uni-bielefeld.de}
\begin{document}
\begin{abstract}
Multilinear representability extends classical linear representability of matroids by assigning subspaces, rather than vectors, to ground elements. This notion is closely related to almost affine codes. In this paper, we introduce and study a $q$-analogue of multilinear representability for $q$-matroids, motivated by known connections between $q$-matroids, classical matroids, and rank-metric codes. We define $m$-multilinear representability in terms of almost affine matrix rank-metric codes satisfying a natural divisibility condition. We prove that nontrivial uniform $q$-matroids admit no purely multilinear representations, and we derive necessary conditions for multilinear representations of almost uniform $q$-matroids.  We further show that the non-Pappus $q$-matroid, if multilinearly representable, must have block size at least $9$. Finally, we prove that no rank-$2$ $q$-matroid on $\F_2^4$ admits a purely $m$-multilinear representation for $1<m<4$, and we classify pure multilinearity for all $q$-matroids on $\F_2^3$ and $\F_2^4$ in the corresponding ranges. At present, no example is known of a purely multilinear $q$-matroid.
\end{abstract}

\maketitle

\section{Introduction}

Matroid theory provides an abstract framework for independence, unifying concepts from linear algebra, graph theory, and combinatorics. A central theme in this theory is \emph{representability}: a matroid is representable over a field if its rank function can be realized by the dimensions of spans of vectors over that field. More precisely, let $M=(E,r)$ be a matroid on a finite ground set $E$. $M$ is $\F_q$-\emph{representable} if there exists a matrix $A$ over $\F_q$, whose columns are indexed by $E$, such that
\[
r(X)=\dim_{\F_q}\big( \langle A_e : e\in X\rangle \big)
\quad \text{for all } X\subseteq E,
\]
where $A_e$ denotes the column of $A$ indexed by $e$. This notion has been extensively studied and has led to deep structural and algorithmic results; see, for instance, \cite{oxley2011matroid}.

A broader notion is \emph{multilinear representability}. Instead of assigning a vector to each ground element, one assigns a subspace of fixed dimension. The rank of a set is then recovered as the dimension of the sum of the corresponding subspaces, normalized by the common block dimension. This generalization is closely related to almost affine codes \cite{simonis1998almost} and appears naturally in coding theory, network coding, and polymatroid theory; see also \cite{pendavingh2013skew, kuhne2023staudt, beimel2014multi, kuhne2026entropic, el2010index}. In \cite{el2010index}, the authors showed that linear network coding capacity problems are equivalent to questions about multilinear matroid representability. This connection establishes multilinear representability as a natural combinatorial abstraction of linear vector network coding. Moreover, it was shown in \cite{kuhne2022representability} that deciding whether a given network coding instance admits a linear vector coding solution is undecidable.

We recall the definition of multilinear representability for classical matroids. Let $m\ge 1$ be an integer.

\begin{definition}
A matroid $M=(E,r)$ is said to be \emph{$m$-multilinear over $\F_q$} if there exists a finite-dimensional vector space $V$ over $\F_q$ and, for each element $e\in E$, a subspace $V_e\subseteq V$ such that
\[
r(X)=\frac{1}{m}\dim_{\F_q}\left( \sum_{e\in X} V_e \right)
\quad \text{for all } X\subseteq E.
\]
\end{definition}

The case $m=1$ recovers ordinary linear representability, after identifying nonzero vectors with the one-dimensional subspaces they span. For $m>1$, one obtains a strictly larger class of matroids. Intuitively, instead of assigning a single vector to each element, one assigns an $m$-dimensional ``block'' subspace, and independence corresponds to direct-sum behavior of these subspaces.

Multilinear matroids are closely related to \emph{almost affine codes} endowed with the Hamming metric. Let $\mC\subseteq \F_q^n$ be a not necessarily linear Hamming-metric code. For $X\subseteq \{1,\dots,n\}$, denote by $\mC_X$ the projection of $\mC$ onto the coordinates in $X$.

\begin{definition}\label{def:alm_affine_classical}
Let $\mC$ be an $\F_q$-linear subspace of $(\Fq^m)^n$. Then $\C$ is an \emph{almost affine code} of length $n$ if and only if, for every subset
$X\subseteq \{1,\dots,n\}$, the dimension $\dim_{\F_q}(\C_X)$ is divisible
by~$m$.
\end{definition}

In this case, the function
\[
r(X):=\frac{1}{m}\log_q |\mC_X|
\]
is integer-valued and defines the rank function of an $m$-multilinear matroid on $\{1,\dots,n\}$; see \cite[Example 2]{simonis1998almost}. If $\mC$ is a linear code over $\F_q$, then for every $X \subseteq \{1,\dots,n\}$ one has
\[
|\mC_X| = q^{\dim_{\F_q}(\mC_X)},
\]
so every linear code is almost affine. In this case, the associated matroid is $\F_q$-representable. Thus, almost affine codes form a wider class than linear codes: they need not carry any linear or additive structure, but they still determine matroids through the cardinalities of their coordinate projections.

A classical example in this context is the \emph{non-Pappus matroid}. It is a rank-$3$ matroid on the ground set $E=\{1,\dots,9\}$ whose $3$-element independent sets are precisely those not contained in one of the lines of the non-Pappus configuration shown in Figure~\ref{fig:non-Pappus}.

\begin{figure}[h!]
\centering
\begin{tikzpicture}[
    scale=1.5,
    point/.style={circle, fill=black, inner sep=2pt}
]

% Points
\node[point,label=above:$1$] (p1) at (0,2) {};
\node[point,label=above:$2$] (p2) at (2,2) {};
\node[point,label=above:$3$] (p3) at (4,2) {};

\node[point,label=below:$4$] (p4) at (0,0) {};
\node[point,label=below:$5$] (p5) at (2,0) {};
\node[point,label=below:$6$] (p6) at (4,0) {};

\node[point,label=left:$7$]  (p7) at (1,1) {};
\node[point,label=above:$8$] (p8) at (2,1) {};
\node[point,label=right:$9$] (p9) at (3,1) {};

% Lines
\draw (p1) -- (p2) -- (p3);
\draw (p4) -- (p5) -- (p6);

\draw (p1) -- (p5) -- (p7);
\draw (p2) -- (p4) -- (p7);

\draw (p3) -- (p5) -- (p9);
\draw (p2) -- (p6) -- (p9);

\draw (p1) -- (p6) -- (p8);
\draw (p3) -- (p4) -- (p8);

\end{tikzpicture}
\caption{Non-Pappus configuration.}
\label{fig:non-Pappus}
\end{figure}
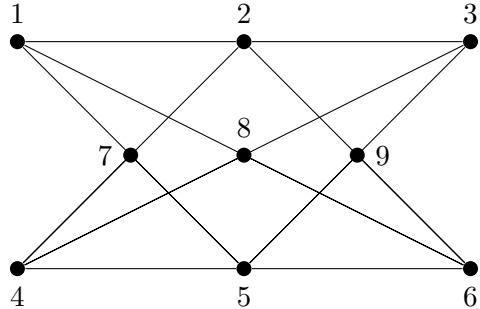

A multilinear representation of this matroid can be constructed as follows. Let $F=(\F_3)^2$, so that the block size is $m=2$, and consider the code $\mC \subseteq F^9$ defined as the $\F_3$-row space of the matrix
\[
\begin{pmatrix}
10 & 10 & 00 & 10 & 00 & 10 & 10 & 10 & 00 \\
01 & 01 & 00 & 01 & 00 & 01 & 01 & 01 & 00 \\
00 & 00 & 00 & 10 & 10 & 21 & 01 & 10 & 10 \\
00 & 00 & 00 & 02 & 01 & 20 & 12 & 02 & 01 \\
00 & 10 & 10 & 01 & 00 & 01 & 00 & 11 & 10 \\
00 & 01 & 01 & 21 & 00 & 21 & 00 & 10 & 01
\end{pmatrix},
\]
where each entry $ab$ denotes the vector $(a,b) \in (\F_3)^2$.

By direct computation, for every subset of coordinates
$X \subseteq \{1,\dots,9\}$, the projection $\mC_X$ has $\F_3$-dimension
divisible by $2$. Equivalently, there exists an integer $r(X)$ such that
\[
|\mC_X| = 3^{2r(X)}.
\]
Thus
\[
r(X)=\frac{1}{2}\dim_{\F_3}(\mC_X)
\]
is integer-valued. The matroid associated with this rank function is the
non-Pappus matroid. In particular, $\mC$ is almost affine and $2$-multilinear, and it yields an explicit example of a $2$-multilinear matroid that is not representable over any field; see \cite{simonis1998almost, oxley2011matroid}.

In recent years, the notions of matroids and polymatroids have been extended to their $q$-analogues, particularly in connection with rank-metric codes; see \cite{jurrius2018defining, gluesing2022q, byrne2024weighted, bcj2022} and the references therein. A $q$-matroid consists of the lattice of subspaces of $\F_q^n$ together with an integer-valued rank function satisfying natural analogues of the usual matroid rank axioms.

In this paper, we investigate a $q$-analogue of multilinear representability. Our approach is motivated by the connection between $q$-matroids and classical matroids established in \cite[Theorem~4.4]{gluesing2022q}. This result associates an ordinary matroid to a $q$-matroid after choosing a basis of the ambient vector space, and it has been used to transfer non-representability results from classical matroids to $q$-matroids. In particular, it applies to the $q$-analogue of the non-Pappus matroid. Since the classical non-Pappus matroid is multilinearly representable, it is natural to ask whether its $q$-analogue admits a corresponding multilinear representation.

We define $m$-multilinear representability for $q$-matroids in terms of $\F_q$-linear matrix rank-metric codes in $\F_q^{n\times m}$. More precisely, we say that a $q$-matroid is $m$-multilinearly representable if it is $\F_q^{n\times m}$-representable, that is, if it is represented by an $\F_q$-linear space of $n\times m$ matrices; see Definition~\ref{def:representability}. This notion generalizes $\F_{q^m}$-representability, where one represents a $q$-matroid by an $\F_{q^m}$-linear rank-metric code in $\F_{q^m}^n$. Hence every $\F_{q^m}$-representable $q$-matroid is $m$-multilinear. We are especially interested in the \emph{purely multilinear} case, namely $q$-matroids that are $\F_q^{n\times m}$-representable but not $\F_{q^m}$-representable. At present, no example of such a $q$-matroid is known.

Our first contribution is a tensorial description of the rank function of a representable $q$-polymatroid. Given a matrix rank-metric code and a generator tensor for it, we express the associated $q$-rank function in terms of slice spaces of this tensor. This description allows us to relate representability of $q$-polymatroids to representability of their projectivizations. In particular, we show that the projectivization of an $m$-multilinear $q$-matroid is an $m$-multilinear classical matroid; see Section \ref{sec:tensor_proj}.
We then investigate pure multilinearity for several families of $q$-matroids. We prove that nontrivial uniform $q$-matroids admit no purely multilinear representations, and we obtain necessary conditions for multilinear representations of almost uniform $q$-matroids; see Sections~\ref{subsec: uniform} and~\ref{subsec: almost_uniform}. We also study the $q$-analogue of the non-Pappus matroid. In this case, we show that if such a $q$-matroid admits an $m$-multilinear representation, then necessarily $m \geq 9$; see Section~\ref{sec: q_non_pappus}. Moreover, we prove that no rank-$2$ $q$-matroid over $\F_2^4$ admits a purely $m$-multilinear representation for $1<m<4$; see Section~\ref{sec: rank2_classification}. Combining this with the rank-one case and duality, we obtain a classification of pure multilinearity for all $q$-matroids on $\F_2^3$ and $\F_2^4$ in the corresponding parameter ranges. We conclude in Section~\ref{sec: conclusion} with open questions and directions for future work.
%%%%%%%%%%%%%%%%%%%%%%%%%%%%%%%%%%%%%%%%%%%%%%%%%%%%%%%%%%%%%%%%%%%

\section{\emph{q}-polymatroids and rank-metric codes}\label{sec:preliminaries}

In this section, we explain the relation between rank-metric codes and $q$-polymatroids. Throughout the paper, we use the following notation.

\begin{notation}\label{not:standing_notation}
Let $q$ be a prime power and let $\F_q$ be the finite field with $q$ elements. Let $n\geq 2$ be a positive integer and set $E:=\F_q^n$. For a space $X$, we denote by $\mL(X)$ the collection of subspaces of $X$. For $0\leq i\leq \dim(X)$, we denote by $\mL(X)_{\leq i}$ and $\mL(X)_{\geq i}$ the collections of all subspaces of~$X$ of dimension at most $i$ and at least $i$, respectively, while we denote by $\mL(X)_{i}$ the collection of all $i$-dimensional subspaces of $X$. 
We write $X\subseteq Y$ to indicate that $X$ is a subspace of $Y$. If a subspace is to be understood as one-dimensional, we represent it by a lowercase letter; for instance, $x \subseteq X$ means that $x$ is a one-dimensional subspace of~$X$. Finally, for an integer $N\in\mathbb{N}$, we write $[N]:=\{1,\ldots,N\}$. Throughout the paper $k$ and $m$ always denotes positive integers.
\end{notation}

\subsection{\emph{q}-Polymatroids}

The first results on $q$-polymatroids and their relation to coding theory can be found in \cite{gorla2019rank} and~\cite{shiromoto2019codes}. The following definition is from \cite[Def.~4.1]{gorla2019rank}.

\begin{definition}\label{def:rank_axioms}
A function $\rho: \mL(E)\longrightarrow \mathbb{R}$ is a \textbf{$q$-rank function} if it satisfies the following axioms.
\begin{itemize}
	\item[(R1)] \emph{Boundedness}: $0\leq \rho(A) \leq \dim(A)$, for all $A \in \mL(E)$.
	\item[(R2)] \emph{Monotonicity}: $A\leq B \Rightarrow \rho(A)\leq \rho(B)$,  for all $A,B \in \mL(E)$.
	\item[(R3)] \emph{Submodularity}: $\rho(A \cap B)+\rho(A+B)\leq \rho(A) +\rho(B)$, for all $A,B \in \mL(E)$.
\end{itemize}
A \textbf{$q$-polymatroid} is a pair $\mM=(\mL(E),\rho)$ for which $\rho$ is a $q$-rank function.
\end{definition}
The value $\rho(E)$ is called the \textbf{rank of $\mM$} and it is also denoted by $\rho(\mM)$. If $\rho(V)\in\mathbb{Z}$ for every $V\in\mL(E)$, then we say that $\mM$ is a \textbf{$q$-matroid}.

\begin{definition}
Let $E_1,E_2$ be two $n$-dimensional vector spaces over $\F_q$.
Two $q$-polymatroids $\mM_i = (\mL(E_i), \rho_i)$, $i = 1, 2$, are \textbf{scaling-equivalent} if there exists an $\F_q$-isomorphism $\alpha\in\mathrm{Hom}(E_1,E_2)$ and $a \in \mathbb{Q}_{>0}$ such that 
$\rho_2(\alpha(V)) = a\rho_1(V)$ for all $V\in\mL(E_1)$. In this case, we write $\mM_1\cong\mM_2$.
\end{definition}

We also have a notion of \emph{duality} for $q$-polymatroids.

\begin{definition}\label{def: qPM_dual}

    Let $\mM = (\mL(E), \rho)$ be a $q$-polymatroid and set 
    \begin{equation}
        \rho^*(V) = \dim(V) + \rho(V^{\perp}) - \rho(E).
    \end{equation}
    Then $\rho^*$ is a $q$-rank function on $E$ and $\mM^* = (\mL(E), \rho^*)$ is a $q$-polymatroid. It is called the \textbf{dual} of $\mM$.
\end{definition}

Since $\rho(0)=0$ by (R1), the rank of the dual is  $\rho^*(\mM^*)=n-\rho(\mM)$. Moreover, we naturally have $(\mM^*)^*=\mM$.

If $\mM$ is a $q$-matroid, then a subspace $A\in\mL(E)$ is called \textbf{independent} if $\rho(A)=\dim(A)$, and \textbf{dependent} otherwise. A dimension-maximal independent space is called a \textbf{basis} of $\mM$. We say that $C \in \mL(E)$ is a {\bf circuit} of $\mM$ if it is itself a dependent space and every proper subspace of $C$ is independent.
Finally, we say that $X\in\mL(E)$ is a \textbf{flat} of $\mM$ if $\rho(X)<\rho(X+ v)$ for all one-dimensional spaces $ v\in\mL(E)_1\setminus \mL(X)_1$ and an inclusion-maximal proper flat is called a \textbf{hyperplane}. These objects are used to provide cryptomorphic characterizations of $q$-matroids; see e.g. \cite{bcj2022, cj2023, jurrius2018defining, alfarano2024cyclic}. In the more general case of $q$-polymatroids, however, most of them fail; see \cite{gluesing2022independent}.

A $q$-matroid $\mM$ is called \textbf{paving} if all its circuits $C$ satisfy $\dim(C)\geq \rho(\mM)$. There exists an explicit construction of paving $q$-matroids, which was illustrated in \cite{gluesing2022q}.

\begin{proposition}(\cite[Proposition~4.6]{gluesing2022q})\label{prop: paving_construction}
    Let $1\leq k\leq n-1$ and let $\mathcal{S}\subseteq\mL(E)_{k}$ be such that
    $\dim(V\cap W)\leq k-2$ for every two distinct $V,W\in\mathcal{S}$. Define the map
    \[        \rho_\mS:\mL(E)\rightarrow\mathbb{Z}_{\geq 0},\quad V\mapsto
        \begin{cases}
            \;\hfil k-1\;&\text{if }V\in\mathcal{S},\\
            \;\hfil \min\{\dim V,k\}\;&\text{otherwise}. 
        \end{cases}
    \]
    Then $\mM_\mS=(\mL(E),\rho_\mS)$ is a paving $q$-matroid of rank $k$. Its circuits of rank $k-1$ are precisely the subspaces in $\mathcal{S}$. We call $\mM_\mS$ the \textbf{paving $q$-matroid induced by $\mS$}.
\end{proposition}

Note that not all paving $q$-matroids arise from the construction in \Cref{prop: paving_construction}. A special example of paving $q$-matroids constructed as in \Cref{prop: paving_construction} is given below.

\begin{example}\label{ex: uniform_qmat}
Let $0\leq k\leq n$. For each $V\in\mL(E)$, define
$\rho(V):=\min\{k,\dim(V)\}$. Then $(\mL(E),\rho)$ is a $q$-matroid.
It is called the \textbf{uniform $q$-matroid on $E$ of rank $k$} and is denoted by $\mU_{k,n}(q)$. It is not difficult to see that, for $1\leq k\leq n-1$, we can regard $\mU_{k,n}(q)$ as $\mM_\emptyset$, i.e., the paving $q$-matroid induced by $\emptyset$.
\end{example}

%%%%%%%%%%%%%%%%%%%%%%%%%%%%%%%%%%%%%%%%%%%%%%%%%%%%%%%%%%%%%%%%%%%

\subsection{Representable \emph{q}-polymatroids}
In this section, we study $q$-polymatroids arising from rank-metric codes; see \cite{gorla2019rank, shiromoto2019codes, gluesing2022q}.
We start by briefly recalling some basic notions on rank-metric codes; for a more detailed treatment, we refer the interested reader to \cite{de78,gabidulin1985theory,gorla2021rank}.
For this purpose, we endow the space of matrices $\F_{q}^{n \times m}$ with the \textbf{rank distance}, defined by $\mathrm{d}(A,B):=\rk(A-B)$, for all $A,B\in\F_{q}^{n \times m}$.

\begin{definition}\label{def:dual}
 We say that $\mC \subseteq \F_q^{n \times m}$ is an {\bf $\F_q$-linear rank-metric code} or a {\bf matrix rank-metric code} if $\mC$ is an $\F_q$-subspace of $\F_{q}^{n \times m}$. Its \textbf{minimum distance} is
 $$\mathrm{d}(\C):=\min\{\rk(M) : M \in \mC, \; M \neq 0\}.$$ We say that $\mC$ is an $\F_q$-$[n \times m, k,d]$ (rank-metric) code if it has $\F_q$-dimension $k$ and minimum distance $d$. 
The \textbf{dual code} of $\mC$ is defined to be $\mC^{\perp}=\{M\in\F_q^{n\times m}: \mathrm{Tr}(MN^\top)=0 \textnormal{ for all } N\in~\mC\}$.
\end{definition} 

We say that two rank-metric codes $\mC_1,\mC_2\subseteq \F_q^{n\times m}$ are \textbf{equivalent} if there exist $(A,B)\in\GL_n(\F_q)\times \GL_m(\F_q)$ such that $\mC_2 = A \mC_1 B$. If $m=n$, then transposition also preserves rank.

The parameters $n,m,k,d$ of a rank-metric code are related by the following inequality known as Singleton bound (see e.g. \cite{de78}):
\begin{equation}\label{eq:Singleton}
    k \leq \max\{n,m\}(\min\{n,m\}-d+1).
\end{equation}
Codes whose parameters meet the Singleton bound with equality are called \textbf{maximum rank distance} codes or \textbf{MRD} codes for short.

It is known that an $\F_q$-linear rank-metric code in $\F_q^{n\times m}$ induces a $q$-polymatroid on $E=\F_q^n$; see \cite{gorla2019rank, shiromoto2019codes}. One way to describe this correspondence is as follows.
\begin{definition}\label{def:codepoly}
Let $\mC$ be an $\F_q$-$[n\times m,k]$ code.
For each subspace $U\in\mL(E)$, we define
$$\mC(U):=\{M \in \C : \colsp_{\F_q}(M) \subseteq U\}.$$ 
It is easy to see that for every $U\in\mL(E)$, $\mC(U)$ is an $\F_q$-subspace of $\mC$, since for every $M,N\in\F_q^{n\times m}$, we have $\colsp_{\F_q}(M+N)\subseteq \colsp_{\F_q}(M) + \colsp_{\F_q}(N)$.
Let 
$$\rho_\mC: \mL (E) \longrightarrow \mathbb{Q}_{\geq 0}, \quad U\longmapsto\frac{k-\dim(\mC(U^\perp))}{m},$$
where $U^\perp$ denotes the orthogonal complement of $U$ with respect to the standard inner product.\footnote{With a slight abuse of notation, we use the symbol ${}^\perp$ both for orthogonal complements of subspaces, with respect to the standard inner product, and for duals of matrix rank-metric codes with respect to the symmetric nondegenerate bilinear form $(M,N)\mapsto \operatorname{Tr}(MN^\top)$.}
Then $(\mL(E),\rho_\mC)$ is a $q$-polymatroid \cite[Theorem 5.3]{gorla2019rank} and we denote it by $\mM[\mC]$.	
\end{definition}

Let $\mC \subseteq \F_q^{n \times m}$ be a rank-metric code and let
$A \in \GL_n(\F_q)$. Let $I \subseteq [n]$ satisfy $0 < |I| < n$. For a matrix $M\in \F_q^{n\times m}$, we denote by $M_I$
the matrix obtained from $M$ by selecting the rows indexed by $I$. We define the \textbf{row-restricted} and \textbf{row-shortened} subcodes of $\mC$ with respect to $A$ and $I$ by
\begin{equation*}
    \Pi_r(\mC, A, I) := \{ (AM)_I : M \in \mC \}, \qquad
    \Sigma_r(\mC, A, I) :=
    \{ (AM)_I : M \in \mC,\; (AM)_{[n]\setminus I} = 0 \}.
\end{equation*}
Similarly, one can define the column-restricted and column-shortened subcodes of $\mC$ with respect to a matrix $B\in\GL_m(\F_q)$ and a set of column indices  $J\subseteq [m]$.

\begin{remark}\label{rem:shortening&C(U)}
Let $\mC \subseteq \F_q^{n \times m}$ be a rank-metric code. The notions of shortened subcode of $\mC$ and of $\mC(U)$ are related. Indeed, let $I \subseteq [n]$ be a set of row indices and let $A \in \GL_n(\F_q)$ be an invertible matrix. Let $\{e_1, \dots, e_n\}$ be the standard basis of $\F_q^n$ and define $V_I = \langle e_i : i \in I\rangle$. Set $U = A^{-1}V_I$, so that $\dim(U)=|I|$.
Then $\Sigma_r(\mC,A,I)\cong \C(U)$. Indeed, for $M\in\mC$, the condition $(AM)_{[n]\setminus I}=0$ is equivalent to
$\operatorname{colsp}(AM)\subseteq V_I$, which is equivalent to
$\operatorname{colsp}(M)\subseteq A^{-1}V_I=U$. Hence the matrices used in the shortening are precisely the elements of $\C(U)$.
\end{remark}

The following result is immediate, but essential for what follows. 

\begin{lemma}
    Let $\mC \subseteq \F_q^{n \times m}$ be a rank-metric code. Let $A \in \GL(n, q)$ and $I \subseteq [n]$. Let $V_{\bar{I}} = \langle e_i \; : \; i\in [n]\setminus I\rangle$ and let $U = A^{-1}V_{\bar{I}}$. Then, the map
    $$\phi^r_{A,I}: \mC \to \Pi_r(\mC, A, I), \quad M \mapsto (AM)_I$$
    is a surjective linear map whose kernel is $\mC(U)$.
\end{lemma}
\begin{proof}
    The map $\phi^r_{A,I}$ is clearly linear and surjective. Moreover, $M \in \C$ lies in $\ker(\phi^r_{A,I})$ if and only if $\phi^r_{A,I}(M) = 0$. This is equivalent to the fact that the matrix $AM$ has zeros in all rows indexed by $I$. Consequently, the non-zero rows of $AM$ can occur only in positions indexed by $[n]\setminus I$. Let $v \in \colsp_{\F_q}(M)$. Then $Av$ is a linear combination of the columns of $AM$. Because every row of $AM$ indexed by $I$ is zero, the vector $Av$ can have nonzero coordinates only  in positions indexed by $[n]\setminus I$. In other words, $Av \in \langle e_i \; : \; i \in [n]\setminus I  \rangle = V_{\bar{I}}$.
    Hence, $v \in A^{-1}V_{\bar{I}} = U$.
    This shows that $\colsp_{\F_q}(M) \subseteq U$. Conversely, if $\colsp_{\F_q}(M) \subseteq U$, then for every column $c$ of $M$, $Ac \in V_{\bar{I}}$, which implies that the rows indexed by $I$ are zero. Thus, $\ker(\phi^r_{A,I}) = \{ M \in \mC : \colsp_{\F_q}(M) \subseteq U \} = \C(U)$.
\end{proof}

As an immediate consequence, let $U\subseteq \F_q^n$ be a subspace of
dimension $u$, and let $A\in \GL(n,q)$ be a change-of-basis matrix that
maps $U$ to the last $u$ standard basis vectors of $\F_q^n$, that is,
$AU=V_{\bar I}$, where $I=\{1,\dots,n-u\}$. Then
\begin{equation}\label{eq_first_iso}
\C/\C(U) \cong \Pi_r(\mC, A, I).
\end{equation}

We now recall the \emph{vector} version of rank-metric codes. Fix an $\F_q$-basis $\Gamma=\{\gamma_1,\ldots,\gamma_m\}$ of $\F_{q^m}$ over $\F_q$. For $x=(x_1,\ldots,x_n)\in\F_{q^m}^n$, let $\Gamma(x)\in\F_q^{n\times m}$ be the matrix obtained by expanding the coordinates of $x$ with respect to $\Gamma$, namely, if $x_j=\sum_{i=1}^m x_{j,i}\gamma_i$ for each $j\in[n]$, then $\Gamma(x)=(x_{j,i})_{j,i}$. We define the \textbf{rank support} of $x$ as $\supp(x):=\colsp_{\F_q}(\Gamma(x))\subseteq \F_q^n$. This definition is independent of the choice of the basis $\Gamma$.

The rank distance on $\F_{q^m}^n$ is defined by $\dd(x,y):=\rk(\Gamma(x)-\Gamma(y))$. This definition is also independent of the choice of
$\Gamma$.

\begin{definition}
An \textbf{$\F_{q^m}$-linear rank-metric code}, or a \textbf{vector rank-metric code}, is an $\F_{q^m}$-linear subspace $\mC\subseteq \F_{q^m}^n$, endowed with the rank distance. If $\dim_{\F_{q^m}}(\mC)=K$ and the minimum rank distance of $\mC$ is $d$, then we say that $\mC$ is an $\F_{q^m}$-$[n,K,d]$ rank-metric code.
\end{definition}

Thus, after fixing $\Gamma$, we may identify $\F_{q^m}^n$ with $\F_q^{n\times m}$. Every $\F_{q^m}$-linear rank-metric code can therefore be viewed as an $\F_q$-linear matrix rank-metric code. Such codes give rise to $q$-matroids as follows.

\begin{definition}\label{def:vector_codes}
Let $\mC$ be a $K$-dimensional $\F_{q^m}$-linear subspace of $\F_{q^m}^n$. For every $W \subseteq \Fq^n$, we define
$\mC(W):=\{x \in \mC : \supp(x) \subseteq W\}$.
Let $\rho: \mL(E) \longrightarrow \mathbb{Z}_{\geq 0}$ be defined by
$$\rho(W):=K-\dim_{\F_{q^m}}(\C(W^\perp)).$$
Then $(\mL(E),\rho)$ is a $q$-matroid \cite[Theorem~24]{jurrius2018defining}, and we also denote it by $\mM[\C]$. Note that in this case, $K$ is the $\F_{q^m}$-dimension of~$\mC$.
\end{definition}

 We have the following notions of representability, which extend \cite[Definition~4.1]{gluesing2022q}. 

\begin{definition}\label{def:representability}
Let $E$ be an $n$-dimensional $\F_q$-vector space and let $\mM = (\mL(E),\rho)$ be a $q$-polymatroid.
\begin{enumerate}
    \item $\mM$ is said to be \textbf{$\F_q^{n\times m}$-representable} if there exists a rank-metric code $\mC \subseteq \F_q^{n \times m}$ such that $\mM = \mM[\C]$.
    
    \item Suppose $\mM$ is a $q$-matroid. Then $\mM$ is said to be \textbf{$\F_{q^m}$-representable} if there exists an $\F_{q^m}$-linear rank-metric code $\mC \subseteq \F_{q^m}^{n}$ such that $\mM = \mM[\C]$. 
    \item Suppose $\mM$ is a $q$-matroid. If $\mM$ is  {$\F_q^{n\times m}$-representable} for some  $m>1$, then we say that $\mM$ is $m$-\textbf{multilinear} (or simply \textbf{multilinear} if the $m$ is clear from the context).
\end{enumerate}
\end{definition}

\begin{remark}\label{rem: multilin_case_m=1}
For $m=1$, the notion of $\F_q^{n\times 1}$-representability coincides, under the natural identification $\F_q^{n\times 1}\cong \F_q^n$, with $\F_q$-representability in the sense of \Cref{def:representability}(2). Therefore, we do not use the term $1$-multilinearity.
\end{remark}

\begin{example}\label{ex: uniform_representability}
The \emph{trivial} uniform $q$-matroid $\mU_{0,n}(q)$ and the \emph{free} uniform $q$-matroid $\mU_{n,n}(q)$ are
$\F_q$-representable. In particular, $\mU_{0,n}(q)$ is represented by the $1 \times n$ zero matrix and $\mU_{n,n}(q)$ by
the $n\times n$ identity matrix. For $0 < k < n$, the uniform $q$-matroid $\mU_{k,n}(q)$ is $\F_{q^m}$-representable if and only if
$m \geq n$. Indeed, a matrix $G\in\F_{q^m}^{k\times n}$ represents $\mU_{k,n}(q)$ if and only if $\rk(GY^\top) = k$ for all $Y \in \F_q^{k\times n}$ of rank $k$.
But this is equivalent to $G$ generating an MRD $\F_{q^m}$-linear code and such a matrix $G$ exists if and only if $m\geq n$; see e.g. \cite[Example 2.4]{gluesing2022representability}. 
\end{example}

Definition~\ref{def:representability}(3) says that if $\mM=(\mL(E),\rho)$ is $m$-multilinear, then it is represented by an $\F_q$-$[n\times m,k]$ rank-metric code $\C$ such that $k$ and $\dim(\C(U))$ are divisible by $m$ for every $U\in\mL(E)$. This leads to the following definition, which can be considered the $q$-analogue of almost affine Hamming-metric codes (see Definition \ref{def:alm_affine_classical}).

\begin{definition}\label{def:almost_affine_rank-metric}
We say that an $\F_q$-$[n\times m,k]$ rank-metric code $\C$ is $m$-\textbf{almost affine} if $\dim_{\F_q}(\C(U))$ is a multiple of $m$ for every $U\subseteq \F_q^n$. In particular, $m$ divides $k$. If the parameter $m$ is clear from the context we simply say that $\C$ is an almost affine rank-metric code.
\end{definition}

From Definition~\ref{def:almost_affine_rank-metric}, it follows that the
$q$-polymatroid arising from an $m$-almost affine rank-metric code is indeed a $q$-matroid, since its rank function is integer-valued. 

Clearly, after fixing an $\F_q$-basis of $\F_{q^m}$, every $\F_{q^m}$-linear rank-metric code is $m$-almost affine, and every
$\F_{q^m}$-representable $q$-matroid is also $m$-multilinearly representable. In this work, we are interested in $m$-multilinear $q$-matroids arising from $m$-almost affine rank-metric codes that are not $\F_{q^m}$-linear. In other words, we are interested in $q$-matroids that are $\F_q^{n\times m}$-representable but not $\F_{q^m}$-representable. In the following, we refer to such $q$-matroids as \textbf{purely $m$-multilinear}.

The notions of representability in \Cref{def:representability} are compatible with $q$-polymatroid duality, as defined in \Cref{def: qPM_dual}. The following theorem makes this compatibility explicit by relating the dual of the $q$-polymatroid associated with a matrix rank-metric code to the $q$-polymatroid associated with
the dual code.

\begin{theorem}(\cite[Theorem~8.1]{gorla2019rank})
Let $\mC\subseteq\F_q^{n\times m}$ be an $\F_q$-linear rank-metric code, and let $\mC^\perp\subseteq\F_q^{n\times m}$ be its dual. Then
$\mM[\mC]^*=\mM[\mC^\perp]$.
\end{theorem}

It follows that the duality bijection between $q$-polymatroids on $E$ of rank $k$ and $q$-polymatroids on $E$ of rank $n-k$ preserves representability in each of the senses considered
in \Cref{def:representability}. Consequently, in \Cref{subsec: classification_F23_F24}, we will restrict ourselves to consider multilinear $q$-matroids of rank $k$ with $0\leq k\leq \left\lfloor \frac{n}{2}\right\rfloor$, where $n=\dim E$.

We conclude this subsection by recalling a construction that associates an ordinary matroid with a given $q$-matroid. This construction is compatible with representability of the $q$-matroid in the sense of \Cref{def:representability}(2). The following result was proved in \cite{gluesing2022q}.

\begin{theorem}(\cite[Theorem 4.4]{gluesing2022q})\label{thm: induced_mat}
    Let $\mM=(\mL(E),\rho)$ be a $q$-matroid, and let
    $\mB=\{v_1,\ldots,v_n\}$ be a basis of $E$. Define
    $$
        r:2^\mB\rightarrow\mathbb{Z},\qquad A\mapsto\rho(\langle A\rangle).
    $$
    Then $(\mB,r)$ is a matroid, called the \textbf{induced matroid} of $\mM$ with respect to $\mB$, and is denoted by $M(\mM,\mB)$. Moreover, if $\mM$ is $\F_{q^m}$-representable, then $M(\mM,\mB)$ is also $\F_{q^m}$-representable.
\end{theorem}

\begin{remark}
    The converse of \Cref{thm: induced_mat} does not hold in general. A notable counterexample is the $q$-matroid constructed in \cite[Section 3.3]{cj2022}, which was shown there to be non-$\F_{q^m}$-representable. However, its induced matroid has ground set of size $4$, and is representable over every field of size at least $3$.
\end{remark}

By applying the contrapositive of the final assertion in \Cref{thm: induced_mat}, the authors of \cite{gluesing2022q} used non-$\F_{q^m}$-representability of induced matroids to prove non-$\F_{q^m}$-representability of certain $q$-matroids. Examples include the Vámos $q$-matroid, the Fano $q$-matroid, and the non-Pappus $q$-matroid (see also \Cref{sec: q_non_pappus}), whose induced matroids are respectively the Vámos, Fano, and non-Pappus matroids. For further details, see \cite[Example 4.8]{gluesing2022q}.

%%%%%%%%%%%%%%%%%%%%%%%%%%%%%%%%%%%%%%%%%%%%%%%%%%%%%%%%%%%%%%%%%%%%%%%

\subsection{3-tensor codes}\label{sub:3tensors}
The theory of rank-metric codes associated with their generating $3$-tensors was developed in \cite{byrne2019tensor}. We recall here some basic results that will be used throughout.

We mainly work with tensor products of the form $\mathbb{F}_q^k \otimes \mathbb{F}_q^n \otimes \mathbb{F}_q^m$, whose elements are called $3$-tensors, third-order tensors, or triads. Throughout, we identify this tensor product with $\mathbb{F}_q^{k \times n \times m}$.
As in \cite{byrne2019tensor}, we introduce the following maps, which define multiplication of $3$-tensors with vectors and matrices:
\begin{align*}
    m_1: \F_q^{s\times k} \times \F_q^{k\times n\times m} &\to \F_q^{s\times n\times m} : (A, X) \mapsto m_1(A, X) = \sum_i (Au_i) \otimes v_i \otimes w_i,\\
    m_2: \F_q^{s\times n} \times \F_q^{k\times n\times m} &\to \F_q^{k\times s\times m} : (B, X) \mapsto m_2(B, X) = \sum_i u_i \otimes (Bv_i) \otimes w_i,\\
    m_3: \F_q^{s\times m} \times \F_q^{k\times n\times m} &\to \F_q^{k\times n\times s} : (D, X) \mapsto m_3(D, X) = \sum_i u_i \otimes v_i \otimes (Dw_i),
\end{align*}
for any $X =\sum_i u_i \otimes v_i \otimes w_i \in\F_q^{k\times n \times m}$.

\begin{remark}
    When $s = 1$, the above maps $m_i$ yield 2-tensors, for $i=1,2,3$. In this case, for ease of notation, we will consider the images of the $m_i$ to be in the spaces of matrices over~$\F_q$.
\end{remark}

\begin{definition}
    Let $X\in\F_q^{N_1}\otimes\F_q^{N_2}\otimes\F_q^{N_3}$. For each $i\in\{1, 2, 3\}$, we define the \textbf{$i$-th slice space} of $X$ to be the $\F_q$-span of $\{m_i(e_j , X) : 1 \leq j \leq N_i\}$, that is,
$$
\sss_i(X) := \langle m_i(e_1,X),\ldots, m_i(e_{N_i}, X)\rangle,
$$
where $e_j$ denotes the $j$-th vector of the standard basis, regarded as a row vector in the appropriate space. We write $\dim_i(X)$ to denote the dimension of $\sss_i(X)$ as an $\F_q$-vector space. We say that $\sss_i(X)$ is \textbf{nondegenerate} if $\dim_i(X) = N_i$, in which case we say that $X$ is \textbf{$i$-nondegenerate}.
\end{definition}

Figures~\ref{fig:slice1}, \ref{fig:slice2}, and \ref{fig:slice3} show the different types of slices of a $3$-tensor obtained by fixing the first, second, and third index, respectively.

 \begin{figure}[htb!]
   \centering
       \begin{minipage}{0.32\textwidth}
       \centering
           
\begin{tikzpicture}[scale=0.7,
    slice/.style={
        draw=black!60,
        line join=round,
        line cap=round,
        very thin
    }
]

% --- Colors ---
\definecolor{topcol}{RGB}{255,242,190}
\definecolor{sideA}{RGB}{238,218,150}
\definecolor{sideB}{RGB}{222,200,125}

% --- Geometry parameters ---
\def\w{4}        % width
\def\d{1.6}      % depth
\def\s{1}        % perspective slope
\def\t{0.35}     % thickness
\def\gap{0.8}    % vertical spacing

% --- Stack ---
\foreach \i in {0,...,4}{
\begin{scope}[yshift=\i*\gap cm]

% front
\path[slice,
      shading=axis,
      left color=sideA!90,
      right color=sideA!60]
(0,0) -- (\w,0) -- (\w,\t) -- (0,\t) -- cycle;

% right side
\path[slice,
      shading=axis,
      left color=sideB!80,
      right color=sideB!60]
(\w,0) -- (\w+\d,\s) -- (\w+\d,\s+\t) -- (\w,\t) -- cycle;

% top
\path[slice,
      shading=axis,
      top color=topcol,
      bottom color=topcol!85]
(0,\t) -- (\w,\t) -- (\w+\d,\s+\t) -- (\d,\s+\t) -- cycle;

\end{scope}
}

\end{tikzpicture}
\caption{\label{fig:slice1}1st slice space.}
\end{minipage}
       \begin{minipage}{0.32\textwidth}
       \centering
           \includegraphics[width=0.8\linewidth]{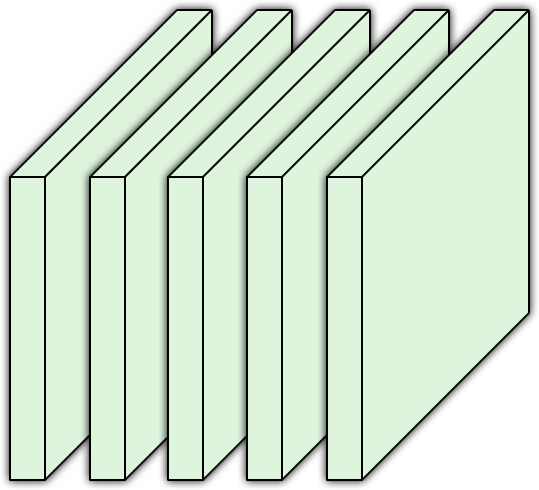}
            \caption{\label{fig:slice2}2nd slice space.}
       \end{minipage}
       \begin{minipage}{0.32\textwidth}
       \centering
           \includegraphics[width=0.8\linewidth]{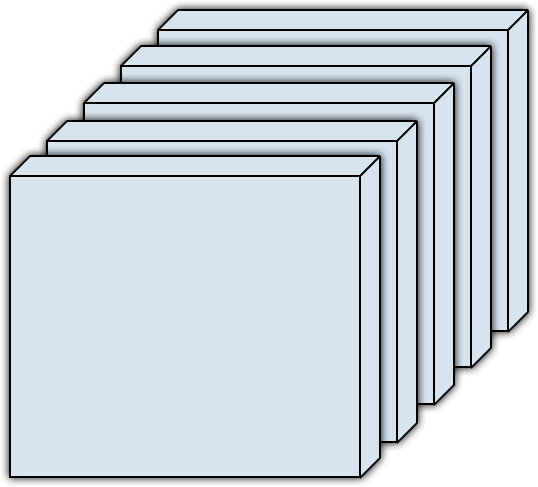}
           \caption{\label{fig:slice3}3rd slice space.}
       \end{minipage}
   \end{figure}

In the theory of rank-metric codes, a natural representation of an $\F_q$-$[n\times m,k]$ code is given by a \emph{generator tensor}. This provides an analogue of the notion of a generator matrix in the rank-metric setting, in the form of a $3$-tensor; for more details we refer the interested reader to~\cite{byrne2019tensor}.

\begin{definition}
Let $\mC$ be an $\F_q$-$[n\times m, k]$ code. A \textbf{generator tensor} for the code $\mC$ is an
element $T\in\F_q^{k\times n\times m}$ such that $\sss_1(T) = \mC$. 
\end{definition}

Let $\mC$ be an $\F_q$-$[n\times m, k]$ code with generator tensor $T\in\F_q^{k\times n\times m}$. Then, $\dim_1(T) = \dim_{\F_q}(\mC)$. However, $\dim_2(T)$ and $\dim_3(T)$ also play an important role, as shown in the following result. 

\begin{lemma}(\cite[Proposition~4.6]{byrne2019tensor})\label{lem:nondegeneracy}
Let $\mC$ be an $\F_q$-$[n\times m, k]$ code with generator tensor $T\in\F_q^{k\times n\times m}$. Then
$$
\dim_2(T) = \dim\left(\sum_{M\in\mC} \colsp_{\F_q}(M)\right),\qquad \dim_3(T) = \dim\left(\sum_{M\in\mC} \rowsp_{\F_q}(M)\right).
$$
In particular, $\mC$ is nondegenerate if and only if every generator tensor for $\C$ is both $2$-nondegenerate and $3$-nondegenerate.
\end{lemma}

The three slice spaces obtained from a tensor $T\in\F_q^{k\times n\times m}$ are closely related. This relation is explained in~\cite{alfarano2026}. We will make use of this relation in the next section.

%%%%%%%%%%%%%%%%%%%%%%%%%%%%%%%%%%%%%%%%%%%%%%%%%%%%%%%%%%%%%%%%%%%%%%%%%%%%%%%%%%%%%%%%%%
%%%%%%%%%%%%%%%%%%%%%%%%%%%%%%%%%%%%%%%%%%%%%%%%%%%%%%%%%%%%%%%%%%%%%%%%%%%%%%%%%%%%%%%%%%
\section{The rank function of representable \emph{q}-polymatroids}\label{sec:tensor_proj}
In this section, we give an equivalent description of the rank function of a representable $q$-polymatroid in terms of a generator tensor of one of its representations.

Let $\C$ be an $\F_q$-$[n\times m,k]$ rank-metric code, and let $T\in\F_q^{k\times n\times m}$ be a generator tensor for
$\C$, that is, $\sss_1(T)=\C$.
Let $U\subseteq \F_q^n$ be a subspace of dimension $u$. Choose
$I=\{1,\ldots,u\}\subseteq [n]$ and a matrix $A\in\GL_n(\F_q)$ such that $U^\perp=A^{-1}V_{\bar I}$,
where $V_{\bar I}=\langle e_i : i\in [n]\setminus I\rangle$.
By the isomorphism~\eqref{eq_first_iso}, applied to $U^\perp$, we then have $\C/\C(U^\perp)\cong \Pi_r(\C,A,I)$.
Moreover, by \cite[Section~7]{byrne2019tensor},
$$\Pi_r(\C,A,I)=\sss_1(m_2(A_I,T)).$$
This motivates the following definition.

\begin{definition}\label{def:rank_tensor}
Let $\C$ be an $\F_q$-$[n\times m,k]$ code and let $T\in\F_q^{k\times n\times m}$ be a generator tensor of $\C$. For each $U\in\mL(\F_q^n)$, define 
$$\rho_T(U):=\frac{\dim\bigl(\sss_1(m_2(A_I,T))\bigr)}{m},$$
where $I=\{1,\ldots,\dim(U)\}$ and $A\in\GL_n(\F_q)$ is any matrix such that $U^\perp=A^{-1}V_{\bar I}$. In the cases $\dim(U)=0$ and $\dim(U)=n$, we use the natural conventions $I=\emptyset$ and $I=[n]$, respectively.
\end{definition}

The next proposition shows that $\rho_T$ is well-defined and coincides with the rank function of the $q$-polymatroid associated with $\C$.

\begin{proposition}
Let $\C$ be an $\F_q$-$[n\times m,k]$ code and let $T\in\F_q^{k\times n\times m}$ be a generator tensor of~$\C$. Then $\rho_T$ is a $q$-rank function and $(\mL(\F_q^n),\rho_T)=\mM[\C]$.
\end{proposition}

\begin{proof}
Let $U\subseteq \F_q^n$, and set $u:=\dim(U)$. Choose $I=\{1,\ldots,u\}$ and $A\in\GL_n(\F_q)$ such that $U^\perp=A^{-1}V_{\bar I}$. By \cite[Section~7]{byrne2019tensor}, we have $ \Pi_r(\C,A,I)=\sss_1(m_2(A_I,T))$. Taking dimensions and using \eqref{eq_first_iso}, applied to $U^\perp$, we obtain
\begin{align*}
    \dim\bigl(\sss_1(m_2(A_I,T))\bigr) = \dim\bigl(\Pi_r(\C,A,I)\bigr) = \dim\bigl(\C/\C(U^\perp)\bigr) = k-\dim(\C(U^\perp)).
\end{align*}

Therefore
\[
    \rho_T(U)=\frac{k-\dim(\C(U^\perp))}{m}.
\]
By Definition~\ref{def:codepoly}, the right-hand side is precisely the rank function of the $q$-polymatroid $\mM[\C]$. Hence $(\mL(\F_q^n),\rho_T)=\mM[\C]$. In particular, $\rho_T$ is a $q$-rank function.
\end{proof}

\begin{remark}
The formula in \Cref{def:rank_tensor} can be viewed as a tensor version of \cite[Remark~4.3]{gluesing2022q}. Indeed, after fixing an $\F_q$-basis $\Gamma$ of $\F_{q^m}$, we may identify $\F_{q^m}^n$ with $\F_q^{n\times m}$ by expanding each coordinate in the basis $\Gamma$.
Thus, if $\mC\leq \F_q^{n\times m}$ has generator tensor $T\in\F_q^{k\times n\times m}$, the preimages of an $\F_q$-basis of $\mC$ give a matrix $G\in\F_{q^m}^{k\times n}$ whose rows generate the corresponding $\F_q$-subspace of $\F_{q^m}^n$. For a subspace $U\leq \F_q^n$, \cite[Remark~4.3]{gluesing2022q} computes $\rho_\mC(U)$ as $\operatorname{rowrk}_{\F_q}(GY)/m$, where the columns of $Y$ form a basis of $U$. In \Cref{def:rank_tensor}, the same quantity is computed as $\dim_{\F_q}\sss_1(m_2(A_I,T))/m$, after moving $U^\perp$ to standard position. The next example illustrates the two descriptions.
\end{remark}

\begin{example}
Let $q=2$, $n=3$, and $m=2$, and fix an $\F_2$-basis $\{1,\gamma\}$ of
$\F_4$. Let $\mC\leq \F_2^{3\times 2}$ be generated by
\[
M_1=
\begin{pmatrix}
1&0\\
0&0\\
0&0
\end{pmatrix},
\qquad
M_2=
\begin{pmatrix}
0&0\\
1&0\\
0&1
\end{pmatrix},
\qquad
M_3=
\begin{pmatrix}
0&1\\
0&0\\
1&0
\end{pmatrix}.
\]
Let $T\in\F_2^{3\times 3\times 2}$ be the generator tensor whose first slices are $M_1,M_2,M_3$. Take $U=\la e_1,e_2\ra$. Then $U^\perp=\la e_3\ra$, so in \Cref{def:rank_tensor} we may take $I=\{1,2\}$ and $A=I_3$. Hence $\sss_1(m_2(A_I,T))$ contains the restrictions of $M_1,M_2,M_3$ to their first two
rows:
\[
(M_1)_I=
\begin{pmatrix}
1&0\\
0&0
\end{pmatrix},
\qquad
(M_2)_I=
\begin{pmatrix}
0&0\\
1&0
\end{pmatrix},
\qquad
(M_3)_I=
\begin{pmatrix}
0&1\\
0&0
\end{pmatrix}.
\]
These are linearly independent over $\F_2$, so
\[
    \rho_T(U)=\frac{\dim_{\F_2}\sss_1(m_2(A_I,T))}{2}=\frac{3}{2}.
\]

On the other hand, with respect to the basis $\{1,\gamma\}$ of $\F_4$, the matrices $M_1,M_2,M_3$ correspond to
\[
    g_1=(1,0,0),\qquad g_2=(0,1,\gamma),\qquad g_3=(\gamma,0,1).
\]
Thus we may take
\[
G=
\begin{pmatrix}
1&0&0\\
0&1&\gamma\\
\gamma&0&1
\end{pmatrix}.
\]
For $U=\la e_1,e_2\ra\subseteq \F_2^3$, take
\[
Y=
\begin{pmatrix}
1&0\\
0&1\\
0&0
\end{pmatrix}.
\]
Then
\[
GY=
\begin{pmatrix}
1&0\\
0&1\\
\gamma&0
\end{pmatrix},
\]
whose $\F_2$-row rank is $3$. Therefore \cite[Remark~4.3]{gluesing2022q}
also gives
\[
    \rho_\mC(U)=\frac{\operatorname{rowrk}_{\F_2}(GY)}{2}=\frac{3}{2}.
\]
\end{example}

%%%%%%%%%%%%%%%%%%%%%%%%%%%%%%%%%%%%%%%%%%%%%%%%%%%%%%%%%%%%%%%%%%%%%%%%%%%%%%%%%%%%%%%%%%
\subsection{Projectivization}
In \cite{alfarano2022linear}, it is shown how to construct a Hamming-metric code from an $\F_{q^m}$-linear rank-metric code, and how the parameters of the two codes are related. In analogy with this $\F_{q^m}$-linear setting, \cite{alfarano2026} illustrates how to associate an \emph{additive Hamming-metric code} with an $\F_q$-linear rank-metric code. This code is obtained by evaluating a generator tensor on the points of the projective space $\PP(\F_q^n)$, and it encodes the same geometric information as the original rank-metric code. We recall the construction below.

Let $\C$ be an $\F_q$-$[n\times m,k]$ nondegenerate rank-metric code, and let $T\in\F_q^{k\times n\times m}$ be a generator tensor for $\C$, that is, $\sss_1(T)=\C$. Let $N=(q^n-1)/(q-1)$, and choose representatives $u_1,\ldots,u_N\in\F_q^n\setminus\{0\}$ of the points of $\PP(\F_q^n)$. For each $i\in[N]$, define $M_i:=m_2(u_i,T)\in\F_q^{k\times m}$, where $u_i$ is regarded as a row vector. Concatenating these blocks, we obtain the matrix
\begin{equation*}
    G^{\mathrm{AH}}
    :=
    \begin{pmatrix}
        M_1 & M_2 & \cdots & M_N
    \end{pmatrix}
    \in\F_q^{k\times mN}.
\end{equation*}

\begin{definition}[\cite{alfarano2026}]
The \textbf{extended additive Hamming-metric code associated with $\C$} is the $\F_q$-linear code
$\C^{\mathrm{AH}}:=\rowsp_{\F_q}(G^{\mathrm{AH}})\subseteq \F_q^{mN}$.
\end{definition}

Different choices of the representatives $u_i$ lead to equivalent additive codes. Therefore, the equivalence class of $\mC^{\mathrm{AH}}$ depends only on the equivalence class of $\mC$.

We now define the projectivization of a $q$-polymatroid. Let $\mM=(\mL(E),\rho)$ be a $q$-polymatroid on $E$. The \textbf{projectivization} of $\mM$ is the polymatroid on the ground set $\PP(E)$ whose rank function is given, for all $A\subseteq\PP(E)$, by
$$
r_{\Proj(\mM)}(A):=\rho(\langle A\rangle),
$$
where $\langle A\rangle$ denotes the $\F_q$-subspace of $E$ generated by any set of nonzero representatives of the points in $A$.

The extended additive Hamming-metric code $\mC^{\mathrm{AH}}$ naturally induces a polymatroid on $[N]$ with rank function
$$
r_{\mC^{\mathrm{AH}}}(V):=\frac{1}{m}\dim\bigl((\mC^{\mathrm{AH}})_V\bigr),\qquad V\subseteq [N],
$$
where $(\mC^{\mathrm{AH}})_V$ denotes the projection of $\mC^{\mathrm{AH}}$ onto the block coordinates indexed by $V$.

Equivalently, if $V=\{i_1,\ldots,i_s\}\subseteq[N]$, then
$$
r_{\mC^{\mathrm{AH}}}(V)
=
\frac{1}{m}\dim\left(
\rowsp_{\F_q}
\begin{pmatrix}
M_{i_1} & \cdots & M_{i_s}
\end{pmatrix}
\right).
$$

The following result is a generalization of \cite[Theorem~6.13]{jany2023proj_matroid} to the case of $\F_q$-linear rank-metric codes.

\begin{theorem}\label{thm:projectivization_qpolymatroid_AH}
Let $\C$ be an $\F_q$-$[n\times m,k]$ nondegenerate rank-metric code with generator tensor $T\in\F_q^{k\times n\times m}$. Let
$\mM[\C]=(\mL(E),\rho_T)$ be the associated $q$-polymatroid, and let $\C^{\mathrm{AH}}$ be an extended additive Hamming-metric code associated with $\C$.
Then the projectivization $\Proj(\mM[\C])$ is equivalent to the polymatroid induced by $\C^{\mathrm{AH}}$.
\end{theorem}

\begin{proof}
Let $\psi:\PP(\F_q^n)\to [N]$ be the bijection defined by
$\psi(\langle u_i\rangle)=i$ for all $i\in[N]$.

Let $A=\{\langle u_{i_1}\rangle,\ldots,\langle u_{i_s}\rangle\}\subseteq\PP(\F_q^n)$,
and set $U:=\langle u_{i_1},\ldots,u_{i_s}\rangle\subseteq \F_q^n$. By definition of the rank function of the polymatroid induced by $\mC^{\mathrm{AH}}$, we have
\begin{equation*}
r_{\mC^{\mathrm{AH}}}(\psi(A))
=
\frac{1}{m}\dim\left(
\rowsp_{\F_q}
\begin{pmatrix}
m_2(u_{i_1},T) & \cdots & m_2(u_{i_s},T)
\end{pmatrix}
\right).
\end{equation*}

Choose a basis $v_1,\ldots,v_t$ of $U$ among the vectors
$u_{i_1},\ldots,u_{i_s}$, and let $A_U\in\F_q^{t\times n}$ be the matrix whose rows are $v_1,\ldots,v_t$. Since each $u_{i_j}$ is an $\F_q$-linear combination of $v_1,\ldots,v_t$, and the map $u\mapsto m_2(u,T)$ is $\F_q$-linear, the block matrix
$\begin{pmatrix}m_2(u_{i_1},T)&\cdots&m_2(u_{i_s},T)\end{pmatrix}$ is obtained from
$\begin{pmatrix}m_2(v_1,T)&\cdots&m_2(v_t,T)\end{pmatrix}$ by right multiplication with a suitable matrix of rank $mt$. Conversely, since $v_1,\ldots,v_t$ are among the $u_{i_j}$, the latter block matrix is obtained from the former by a block-coordinate projection. Hence, the two block matrices have the same rank.

Therefore
\begin{equation*}
r_{\mC^{\mathrm{AH}}}(\psi(A))
=
\frac{1}{m}\dim\left(
\rowsp_{\F_q}
\begin{pmatrix}
m_2(v_1,T) & \cdots & m_2(v_t,T)
\end{pmatrix}
\right).
\end{equation*}

Since the rows of $A_U$ are $v_1,\ldots,v_t$, the block matrix on the right represents the first slice space of the tensor $m_2(A_U,T)$. Thus
\begin{equation*}
r_{\mC^{\mathrm{AH}}}(\psi(A))
=
\frac{1}{m}\dim\bigl(\sss_1(m_2(A_U,T))\bigr).
\end{equation*}
This equals
$\rho_T(U)=\rho_T(\langle A\rangle)=r_{\Proj(\mM[\C])}(A)$.

Hence $\psi$ preserves the rank function. Therefore $\psi$ induces an isomorphism between $\Proj(\mM[\C])$ and the polymatroid induced by $\C^{\mathrm{AH}}$.
\end{proof}

\begin{corollary}
Let $\mM$ be an $m$-multilinear $q$-matroid admitting a nondegenerate
representation. Then its projectivization $\Proj(\mM)$
is an $m$-multilinear matroid over $\F_q$.
\end{corollary}

\begin{proof}
Let $\mC$ be the $\F_q$-$[n\times m,k]$ nondegenerate rank-metric code representing $\mM$, that is $\mM=\mM[\mC]$.
By Theorem~\ref{thm:projectivization_qpolymatroid_AH}, the projectivization $\Proj(\mM)$ is isomorphic to the polymatroid induced by the extended additive Hamming-metric code $\mC^{\mathrm{AH}}$.

Let $u_1,\ldots,u_N$ be representatives of the points of $\PP(\F_q^n)$ and let $T$ be a generator tensor for~$\mC$. For each $i\in[N]$, set
$$
M_i:=m_2(u_i,T)\in\F_q^{k\times m}.
$$
Then the rank function of the polymatroid induced by $\mC^{\mathrm{AH}}$ is given by
$$
r(V)=\frac{1}{m}\dim\left(
\rowsp_{\F_q}
\begin{pmatrix}
M_{i_1} & \cdots & M_{i_s}
\end{pmatrix}
\right),
\qquad\text{for }V=\{i_1,\ldots,i_s\}\subseteq[N].
$$
Thus $\Proj(\mM)$ is $m$-multilinear over $\F_q$ in the classical sense. Finally, since $\mM$ is a $q$-matroid, its rank function is integer-valued. Therefore the rank function of $\Proj(\mM)$ is integer-valued as well, and hence $\Proj(\mM)$ is a matroid.
\end{proof}

%%%%%%%%%%%%%%%%%%%%%%%%%%%%%%%%%%%%%%%%%%%%%%%%%%%%%%%%%%%%%%%%%%%%%%%%%%%%%%%%%%%%%%%%%%
%%%%%%%%%%%%%%%%%%%%%%%%%%%%%%%%%%%%%%%%%%%%%%%%%%%%%%%%%%%%%%%%%%%%%%%%%%%%%%%%%%%%%%%%%%
\section{Multilinear representations of uniform and almost uniform \emph{q}-matroids}\label{sec: multilinearity-unif}

In this section, we investigate multilinearity for uniform and almost uniform $q$-matroids. We first show that nontrivial uniform $q$-matroids admit no purely multilinear representations. We then study almost uniform $q$-matroids, obtained from the uniform case by removing one arbitrary basis, and derive necessary conditions on the parameters of their matrix-code representations.

%%%%%%%%%%%%%%%%%%%%%%%%%%%%%%%%%%%%%%%%%%%%%%%%%%%%%%%%%%%%%%%%%%%%%%%%%%%%%%%%%%%%%%%%%%
\subsection{Uniform \emph{q}-matroids} \label{subsec: uniform}
Let $E=\F_q^n$, and let $0<k<n$ be an integer. In this subsection, we fix $\mM=\mU_{k,n}(q)$ to be the uniform $q$-matroid of rank $k$ on $E$, with rank function $\rho_\mM$. Equivalently, $\mM$ is the paving $q$-matroid induced by $\mS=\emptyset\subseteq\mL(E)_k$, as in \Cref{prop: paving_construction}.

We want to investigate whether there exists some integer $m>0$ such that
$\mM$ is $\F_q^{n\times m}$-representable but not $\F_{q^m}$-representable; that is, whether $\mM$ is purely $m$-multilinear.

Recall from \Cref{ex: uniform_representability} that, in this case, $\mM$ is $\F_{q^m}$-representable if and only if $m\geq n$. Therefore, in the remainder of this subsection, we assume that $m<n$.

\begin{theorem}
    \label{thm:uniform_not_multilinear}
     Let $\mM=\mU_{k,n}(q)$ with $0<k<n$. Then there is no $m<n$ such that $\mM$ is $\F_q^{n\times m}$-representable.
\end{theorem}

\begin{proof}
    Assume, for a contradiction, that $\mM$ is $\F_q^{n\times m}$-representable for some $m<n$. Then there exists an $\F_q$-$[n\times m,k^\prime,d]$ matrix rank-metric code $\mC$ such that, for all $U\in\mL(E)$, we have  
    \[
        \rho_\mC(U)=\frac{k^\prime-\dim(\C(U^\perp))}{m}=\rho_\mM(U).
    \]

    These equalities imply the following properties:
    \begin{enumerate}
        \item[(1)] The dimension of $\mC$ is $k^\prime=km$.
        \item[(2)] For all $U\in\mL(E)_{\geq k}$, we have $\dim(\C(U^\perp))=0$.
        \item[(3)] For all $U\in\mL(E)_{\leq k-1}$, we have
        $\dim(\C(U^\perp))=(k-\dim(U))m$.
    \end{enumerate}

    Indeed, (1) follows by evaluating the equality at $U=E$, since $\C(E^\perp)=\C(0)=\{0\}$ and $\rho_\mM(E)=k$. Using $k^\prime=km$, property (2) follows from the fact that $\rho_\mM(U)=k$ for all $U\in\mL(E)_{\geq k}$, while property (3) follows from the fact that $\rho_\mM(U)=\dim(U)$ for all $U\in\mL(E)_{\leq k-1}$.

    Property (2) implies that $\mC$ contains no nonzero matrix of rank at most $n-k$. Indeed, if $M\in\mC$ has rank at most $n-k$, then its column space is contained in some subspace of $E$ of dimension $n-k$, say $U^\perp$ with $\dim(U)=k$. This would imply $M\in\C(U^\perp)=\{0\}$, a contradiction.

    On the other hand, by property (3), for every $U\in\mL(E)_{k-1}$ we have $\dim(\C(U^\perp))=m>0$. Hence $\C(U^\perp)$ contains a nonzero matrix whose rank is at most $\dim(U^\perp)=n-k+1$. Since ranks at most $n-k$ are excluded, the  minimum distance of $\mC$ is $d=n-k+1$. In particular, $d\leq m$.

    If $k=1$, then $d=n$, and hence $n\leq m$, contradicting $m<n$. Suppose now that $2\leq k\leq n-1$. Since $m<n$, the Singleton bound gives
    \begin{align*}
        km=k^\prime
        &\leq \max\{n,m\}(\min\{n,m\}-d+1) \\
        &= n(m-(n-k+1)+1) \\
        &= n(m-n+k).
    \end{align*}
    Hence $km\leq nm-n(n-k)$, or equivalently $n(n-k)\leq m(n-k)$. Since $k<n$, this implies $n\leq m$, again contradicting $m<n$.

    Therefore, no such $\F_q$-$[n\times m,km,n-k+1]$ matrix rank-metric code exists for $m<n$. Hence $\mM$ is not $\F_q^{n\times m}$-representable for any $m<n$.
\end{proof}

The following corollary is an immediate consequence of \Cref{thm:uniform_not_multilinear} and \Cref{ex: uniform_representability}.

\begin{corollary}
    For $0<k<n$, the uniform $q$-matroid $\mU_{k,n}(q)$ is not purely $m$-multilinear, for any $m>1$.
\end{corollary}

\begin{remark}
    Another consequence of \Cref{thm:uniform_not_multilinear} and
    \Cref{ex: uniform_representability} is that every matrix-code representation of the uniform $q$-matroid $\mU_{k,n}(q)$, for $0<k<n$, must have $m\geq n$. Moreover, such representations exist for $m\geq n$ via $\F_{q^m}$-linear MRD codes. Thus \Cref{thm:uniform_not_multilinear} shows that $\mU_{k,n}(q)$ admits no purely multilinear representation. This differs from the classical matroid case. Indeed, it is known that the matroid $U_{2,4}$ is not $\F_2$-representable, but it has a $2$-multilinear representation over $\F_2$ given by the matrix
    \begin{equation*}
    G=\begin{pmatrix}
        10 & 00 & 10 & 10 \\
        01 & 00 & 01 & 01 \\
        00 & 10 & 10 & 01\\
        00 & 01 & 01 & 11
    \end{pmatrix}.
    \end{equation*}
    Moreover, $U_{2,4}$ arises as the induced matroid of the $q$-matroid
    $\mU_{2,4}(2)$, in the sense of \cite[Theorem~4.4]{gluesing2022q}; see also Theorem~\ref{thm: induced_mat}. Therefore,
    \Cref{thm:uniform_not_multilinear} provides evidence that multilinearity of an induced matroid of a $q$-matroid does not imply multilinearity of the $q$-matroid itself, similarly to what happens for classical representability; see
    \cite[Theorem~4.4]{gluesing2022q}.
\end{remark}

%%%%%%%%%%%%%%%%%%%%%%%%%%%%%%%%%%%%%%%%%%%%%%%%%%%%%%%%%%%%%%%%%%%%%%%%%%%%%%%%%%%%%%%%%%
\subsection{Almost uniform \emph{q}-matroids}\label{subsec: almost_uniform}

Let $E=\F_q^n$, for $n\geq 1$ and let $0<k<n$.  In this subsection, we consider a class of $q$-matroids that arise from the uniform case by removing one arbitrary basis. We then discuss their multilinearity.

\begin{definition}\label{def: almost_uniform}
    Let $X$ be a $k$-dimensional subspace of $E$ and $\mS=\{X\}$. We call the paving $q$-matroid induced by $\mS$ an \textbf{almost uniform $q$-matroid} and denote it by $\mA\mU_{k,n}(q,X)$. Equivalently, $\mA\mU_{k,n}(q,X)$ is the $q$-matroid whose bases are the elements of $\mL(E)_k\setminus\{X\}$.
\end{definition}

\begin{remark}
The subspace $X$ in \Cref{def: almost_uniform} is a circuit-hyperplane, that is, it is both a circuit and a hyperplane of $\mA\mU_{k,n}(q,X)$. Therefore, $\mA\mU_{k,n}(q,X)$ is a sparse paving $q$-matroid with exactly one circuit-hyperplane, in analogy with the classical notion of an almost uniform matroid. Here, as in the classical setting, \emph{sparse paving} means that both the $q$-matroid and its dual are paving.
\end{remark}

The next theorem determines the possible parameters of a matrix rank-metric code with $m<n$ that could yield a multilinear representation of $\mA\mU_{k,n}(q,X)$.

\begin{theorem}\label{thm: almost_uniform_not_mul}
Let $0<k<n-1$, let $m<n$, and let $\mM=\mA\mU_{k,n}(q,X)$ be an almost uniform $q$-matroid. If $\mM$ is represented by an $\F_q$-$[n\times m,k',d]$ matrix rank-metric code $\mC\subseteq\F_q^{n\times m}$, then $m=n-1$ and $\mC$ has parameters $\F_q$-$[n\times(n-1),k(n-1),n-k]$.
\end{theorem}

\begin{proof}
    We denote by $\rho_\mM$ the $q$-rank function of $\mM=\mA\mU_{k,n}(q,X)$. Let $\mC$ be an $\F_q$-$[n\times m,k^\prime,d]$ matrix rank-metric code such that for all $U\in\mL(E)$ we have  
    \[
        \rho_\mC(U)=\frac{k^\prime-\dim(\C(U^\perp))}{m}=\rho_\mM(U).
    \]

   These equalities imply the following properties:
\begin{enumerate}
    \item[(1)] The dimension of $\mC$ is $k'=km$.
    \item[(2)] For all $U\in\mL(E)_{\geq k}\setminus\{X\}$, we have
    $\dim(\C(U^\perp))=0$.
    \item[(3)] For $U=X$, we have $\dim(\C(X^\perp))=m$.
    \item[(4)] For all $U\in\mL(E)_{\leq k-1}$, we have
    $\dim(\C(U^\perp))=(k-\dim(U))m$.
\end{enumerate}
Indeed, (1) follows by evaluating the equality at $U=E$, since
$\C(E^\perp)=\C(0)=\{0\}$ and $\rho_\mM(E)=k$. Using $k'=km$, property (2) follows from the fact that $\rho_\mM(U)=k$ for all
$U\in\mL(E)_{\geq k}\setminus\{X\}$. Property (3) follows from
$\rho_\mM(X)=k-1$, and property (4) follows from
$\rho_\mM(U)=\dim(U)$ for all $U\in\mL(E)_{\leq k-1}$.

Property (2), applied to subspaces $U\in\mL(E)_{k+1}$, implies that $\mC$ contains no nonzero matrix of rank at most $n-k-1$. Indeed, if $M\in\mC$ has rank at most $n-k-1$, then $\colsp_{\F_q}(M)$ is contained in some subspace of dimension $n-k-1$, say $U^\perp$ with $\dim(U)=k+1$. This gives $M\in\C(U^\perp)=\{0\}$.

On the other hand, by property (3), $\C(X^\perp)$ is nonzero. Hence $\mC$ contains a nonzero matrix whose column space is contained in $X^\perp$, and therefore whose rank is at most $\dim(X^\perp)=n-k$. Since ranks at most $n-k-1$ are excluded, the minimum
distance of $\mC$ is $d=n-k$. In particular, $n-k\leq m<n$.

Now choose a one-dimensional subspace $e\subseteq E$ such that $X\not\subseteq e^\perp$. By \cite[Theorem~5.5]{gluesing2022q}, puncturing $\mC$ with respect to $e$ gives a matrix rank-metric code in $\F_q^{(n-1)\times m}$ representing the restriction of $\mM$ to $e^\perp$. Since $X\not\subseteq e^\perp$, this restriction is the uniform $q$-matroid of rank $k$ on the $(n-1)$-dimensional space $e^\perp$.
Since $0<k<n-1$, \Cref{thm:uniform_not_multilinear} applies to this uniform $q$-matroid. Hence no such punctured matrix rank-metric code exists for $m<n-1$. Together with $m<n$, this forces $m=n-1$.
Therefore $k'=km=k(n-1)$ and $d=n-k$. Hence $\mC$ has parameters $\F_q$-$[n\times(n-1),k(n-1),n-k]$.
\end{proof}

\begin{remark}
The proof of Theorem~\ref{thm: almost_uniform_not_mul} actually shows a little more. If the assumption $m<n$ is removed, the same argument implies that any matrix rank-metric code representation of $\mA\mU_{k,n}(q,X)$, for $0<k<n-1$, must satisfy $m\geq n-1$. The hypothesis $m<n$ is only used to conclude that $m=n-1$. In the range $m\geq n$, however, the argument does not address the existence of such representations, nor whether they can be chosen to be
$\F_{q^m}$-linear.
\end{remark}

%%%%%%%%%%%%%%%%%%%%%%%%%%%%%%%%%%%%%%%%%%%%%%%%%%%%%%%%%%%%%%%%%%%%%%%%%%%%%%%%%%%%%%%%%%
\section{The non-Pappus \emph{q}-matroid}\label{sec: q_non_pappus}

In this section, we discuss the multilinearity of the $q$-analogue of the non-Pappus matroid, which was introduced in \cite{gluesing2022q}. In the classical matroid case, the non-Pappus matroid is known to be multilinear, for instance over $\F_3$ and over the field of rational numbers. It is therefore natural to ask whether its $q$-analogue is multilinear as well. It is known, however, that the non-Pappus $q$-matroid is not $\F_{q^m}$-representable for any $m$. This follows from \Cref{thm: induced_mat},
since its induced matroid is the classical non-Pappus matroid, which is not representable over any field; see \cite[Example~4.8]{gluesing2022q}.

The following definition is taken from~\cite[Section~4]{gluesing2022q} and \Cref{prop: paving_construction}.

\begin{definition}
    Let $E=\F_q^9$ and $\mathbb{B}=\{e_1,\ldots,e_9\}$ be the standard basis of $E$. Consider the set $S$ of $3$-element subsets of $\mathbb{B}$ given by
    \[
       S:=\{123,168,157,247,269,348,359,456\},
    \]
    where, for $1\leq i<j<k\leq 9$, the expression $ijk$ denotes the subset $\{e_i,e_j,e_k\}\subseteq\mathbb{B}$. We define the set $\mS\subseteq\mL(E)_3$ via:
    \[
        \mS:=\{\la A\ra_{\F_q}\;|\; A\in S\}.
    \]
    Then $\mS$ satisfies the condition in \Cref{prop: paving_construction} for $k=3$. We call the resulting paving $q$-matroid of rank $3$ the \textbf{non-Pappus $q$-matroid} and denote it by $\mN\mP(q)$. 
\end{definition}

In particular, the rank function of $\mN\mP(q)$ is given, for all $U\in\mL(E)$, by
\begin{equation}\label{eq:rank_non_pappus}
    \rho(U)=
    \begin{cases}
       \dim(U) & \textnormal{if } U\in\mL(E)_{\leq 3}\setminus \mS,\\
       2 & \textnormal{if } U\in \mS,\\
       3 & \textnormal{otherwise.}
    \end{cases}
\end{equation}

The following result partially answers the above question about the multilinearity of the non-Pappus $q$-matroid.

In order to prove it, we recall how the rank distribution of a code $\mC$ is related to the rank distribution of the dual code $\mC^\perp$ via the so-called MacWilliams-type identities. The following identity was proved in \cite[Theorem~30]{CGLR18} and will be a crucial tool in what follows.

\begin{definition}\label{def: rank_distribution}
   Let $\mC\subseteq\F_q^{n\times m}$ be an $\F_q$-linear rank-metric code. The \textbf{rank distribution} of $\mC$ is the sequence $(A_i(\mC))_{i\in\mathbb{N}}$,
where
    \[
        A_i(\mC):=|\{M\in\mC\;|\;\text{rk}(M)=i\}|,
    \]
    for all $i\in\mathbb{N}$.
\end{definition}

\begin{theorem}\label{thm: MacWil_identity}
Let $\mC\subseteq\F_q^{n\times m}$ be an $\F_q$-linear rank-metric code of dimension $t$, and set $N:=\min\{m,n\}$ and $M:=\max\{m,n\}$. Then for every $0\le r\le N$ we have
$$\sum_{i=0}^{N} q^{(N-i)r}A_i(\C)
=
\sum_{j=0}^{r} A_j(\C^\perp)
\left(
\sum_{\nu=j}^{r}
q^{\,t-M\nu}
\binom{N-j}{\nu-j}_q
\binom{r}{\nu}_q
\prod_{s=0}^{\nu-1}(q^\nu-q^s)
\right).
$$
\end{theorem}

\begin{theorem}\label{thm: qnon_pappus_multilin}
The non-Pappus $q$-matroid is not $m$-multilinear for any $m\leq 8$. 
\end{theorem}

\begin{proof}
Let $\rho_\mM$ be the $q$-rank function of $\mM=\mN\mP(q)$. Suppose, for a contradiction,
that $\mN\mP(q)$ is $m$-multilinear for some $m\leq 8$. Then there exists an $\F_q$-$[9\times m,k',d]$ rank-metric code $\mC$ such that, for all $U\in\mL(E)$, we have
\begin{equation*}
    \rho_\mC(U)=\frac{k'-\dim(\mC(U^\perp))}{m}=\rho_\mM(U).
\end{equation*}

   Using \eqref{eq:rank_non_pappus}, we obtain the following properties:
\begin{enumerate}
    \item[(1)] $\dim(\mC)=k'=3m$.
    \item[(2)] For all $U\in\mL(E)_{\geq 3}\setminus\mS$, we have
    $\dim(\C(U^\perp))=0$.
    \item[(3)] For all $U\in\mS$, we have $\dim(\C(U^\perp))=m$.
    \item[(4)] For all $U\in\mL(E)_{\leq 2}$, we have
    $\dim(\C(U^\perp))=(3-\dim(U))m$.
\end{enumerate}
Indeed, (1) follows by evaluating at $U=E$, since $\rho_\mM(E)=3$ and
$\C(E^\perp)=\C(0)=\{0\}$. Using $k'=3m$, property (2) follows from the fact that $\rho_\mM(U)=3$ for all $U\in\mL(E)_{\geq 3}\setminus\mS$. Property (3) follows from $\rho_\mM(U)=2$ for $U\in\mS$, and property (4) follows from $\rho_\mM(U)=\dim(U)$ for all $U\in\mL(E)_{\leq 2}$.
    
By (2), for every $U\in\mL(E)_4$ we have
$\C(U^\perp)=\{M\in\C:\colsp_{\F_q}(M)\subseteq U^\perp\}=\{0\}$. Hence $\mC$ contains no nonzero matrix of rank at most $5$. Indeed, if $M\in\mC$ had rank at most $5$, then $\colsp_{\F_q}(M)$ would be contained in some $5$-dimensional subspace, say $U^\perp$ with $\dim(U)=4$, contradicting $\C(U^\perp)=\{0\}$.

On the other hand, by (3), for every $U\in\mS$ the space $\C(U^\perp)$ is nonzero.
Since $\dim(U^\perp)=6$, this gives nonzero matrices in $\mC$ of rank at most $6$. As ranks at most $5$ are excluded, the minimum rank distance of $\mC$ is $d=6$. Therefore, $m\geq 6$, and so no $m$-multilinear representation exists for $m\leq 5$.

For $m\in\{6,7\}$, the Singleton bound gives a contradiction. Indeed, since $n=9$ and $m<9$, we obtain
\begin{equation*}
    k'\leq \max\{9,m\}(\min\{9,m\}-d+1)
    \quad\Longrightarrow\quad
    3m\leq 9(m-6+1).
\end{equation*}
For $m=6$, this gives $18\leq 9$, and for $m=7$, it gives $21\leq 18$, both contradictions.

It remains to exclude the case $m=8$. Assume that there exists an $8$-multilinear representation of $\mN\mP(q)$. Then it is given by an $\F_q$-$[9\times 8,24,6]$ rank-metric code $\mC$ satisfying properties (1)--(4).

We compute the rank distribution of $\mC$. Clearly, $A_0(\mC)=1$ and
$A_1(\mC)=\cdots=A_5(\mC)=0$, since the minimum rank distance of $\mC$ is $6$. Moreover, every rank-$6$ matrix $M\in\mC$ has column space equal to $X^\perp$ for some $X\in\mS$. Indeed, if $\rk(M)=6$ and $V=\colsp_{\F_q}(M)$, then
$\dim(V^\perp)=3$, and properties (2) and (3) force $V^\perp\in\mS$.

For each $X\in\mS$, property (3) gives  $\dim(\C(X^\perp))=8$, and all nonzero matrices in $\C(X^\perp)$ have rank $6$. Since distinct spaces in $\mS$ have intersection of dimension at most $1$, the spaces $\C(X^\perp)$ intersect pairwise trivially. Therefore
\begin{equation*}
    A_6(\mC)=8(q^8-1).
\end{equation*}

To compute $A_7(\mC)$, we notice that a matrix of rank $7$ has column space of dimension $7$. We claim that any $7$-dimensional subspace $W\le E$ contains at most one $X^\perp$, with $X\in\mS$.
Indeed, for distinct $X,Y\in\mathcal S$ we have $\dim(X\cap Y)\le 1$, by construction. Hence 
\[
\dim(X^\perp\cap Y^\perp)=\dim((X+Y)^\perp)\le 4.
\]
If $W$ contained both $X^\perp$ and $Y^\perp$, then $W^\perp \subseteq X\cap Y$, which is impossible since $\dim(W^\perp)=2$ while $\dim(X\cap Y)\le 1$. Fix a $7$-dimensional space $W\le E$. Then $\dim \C(W)=8$, by (4), so $|\C(W)|=q^8$.
If $W$ contains no $X^\perp$, for $X\in\mS$, then $\C(W)$ contains no nonzero codeword of rank $\le 6$, hence all nonzero elements of $\C(W)$ have rank $7$. Thus $W$ contributes $q^8-1$ to $A_7(\mC)$.
If $W$ contains some $X^\perp$, for $X\in\mS$, then $\C(W)=\C(X^\perp)$, and all nonzero codewords have rank $6$. Hence $W$ contributes $0$ to $A_7(\mC)$. Then, we have to count the $7$-spaces containing no $X^\perp$, for $X\in\mS$.
The total number of $7$-spaces is $\binom{9}{7}_q=\binom{9}{2}_q.$
For fixed $X\in\mS$, the number of $7$-spaces containing $X^\perp$ is $\binom{3}{1}_q$.
Since there are $8$ such $X$ and no $7$-space contains two of them, the number of $7$-spaces containing some $X^\perp$ is $8\binom{3}{1}_q$.
Therefore, the number of $7$-spaces containing none is $\binom{9}{2}_q-8\binom{3}{1}_q$, and 
$$ A_7(\mC)=(q^8-1)\left(\binom{9}{2}_q-8\binom{3}{1}_q\right).$$
Consequently, any other matrix in $\mC$ must have rank $8$, and hence we obtain that 
$$ A_8(\mC) = q^{24} - A_6(\mC) - A_7(\mC) -1.$$

We now apply the MacWilliams identity from \Cref{thm: MacWil_identity} with $r=1$. Since here $N=8$, $M=9$, and $t=24$, we obtain
\begin{align*}
\sum_{i=0}^{8} q^{8-i}A_i(\mC)
&=
\sum_{j=0}^{1} A_j(\mC^\perp)
\left(
\sum_{\nu=j}^{1}
q^{24-9\nu}
\binom{8-j}{\nu-j}_q
\binom{1}{\nu}_q
\prod_{s=0}^{\nu-1}(q^\nu-q^s)
\right)\\
& = q^{24}+q^{15}(q^8-1)+q^{15}(q-1)A_1(\mC^\perp).
\end{align*}

On the other hand, using the rank distribution of $\mC$ that we have computed above, we obtain
$$
\sum_{i=0}^8 q^{8-i}A_i(\mC)
=
q^{24}+(q^8-1)+(q^2-1)A_6(\mC)+(q-1)A_7(\mC).
$$

Substituting the values of  $A_6(\mC)$ and $A_7(\mC)$, we obtain
$$
q^{15}(q-1)A_1(\mC^\perp)
=
(q^8-1)\Bigl(
8(q^2-1)+1+(q-1)\bigl(\tbinom{9}{2}_q-8\tbinom{3}{1}_q\bigr)-q^{15}
\Bigr).$$
By simplifying, we obtain that the right-hand side equals
$$
(q^8-1)\,q^2(q-1)\,P(q),
$$
where
$$
P(q)=
q^{10}+q^9+2q^8+2q^7+3q^6+3q^5+3q^4+2q^3+2q^2+q-7.$$

Hence
$$
A_1(\mC^\perp)=\frac{(q^8-1)P(q)}{q^{13}}.$$
Since $\gcd(q^8-1,q)=1$, for $A_1(\mC^\perp)$ to be an integer, we need to have $q^{13}\mid P(q)$. We distinguish between two cases.
If $q\neq 7$, then reducing modulo $q$ yields $$P(q)\equiv -7 \pmod q,$$ so $q\nmid P(q)$, a contradiction. If $q=7$, then reducing modulo $7^3$ gives $P(7)\equiv 98 \pmod{7^3}$, so $7^3\nmid P(7)$, and hence $7^{13}\nmid P(7)$. Thus $A_1(\mC^\perp)$ is not an integer, which is impossible.

Therefore no rank-metric code $\mC\subseteq\F_q^{9\times 8}$ satisfying the above conditions exists. Hence $\mN\mP(q)$ is not $8$-multilinear. Combining this with the previous exclusions for $m\leq 7$, we conclude that $\mN\mP(q)$ is not $m$-multilinear for any $m\leq 8$.
\end{proof}

\Cref{thm: qnon_pappus_multilin} shows that, in order to find a multilinear representation of the non-Pappus $q$-matroid, one has to consider $\F_q$-$[9\times m,3m,6]$ rank-metric codes with $m\geq 9$. In the next result, we determine the rank distribution that such a putative code representing $\mN\mP(q)$ must have, if it exists, for any $m\geq 9$.

\begin{theorem}\label{thm: qnonpappus_weight_distribution_mgeq9}
Assume that the non-Pappus $q$-matroid $\mN\mP(q)$ is $m$-multilinear for some $m\geq 9$, and let $\mC$ be an $\F_q$-$[9\times m,3m,d]$ representing matrix rank-metric code. Then $d=6$, and the rank distribution of $\mC$ is given by 
\begin{gather*}
A_0(\mC)=1,\qquad A_1(\mC)=\cdots=A_5(\mC)=0,\\
A_6(\mC)=8(q^m-1),\qquad
A_7(\mC)=(q^m-1)\left(\binom{9}{2}_q-8\binom{3}{1}_q\right),\\
A_8(\mC)=\binom{9}{1}_q(q^{2m}-1)-\binom{3}{1}_qA_6(\mC)-(q+1)A_7(\mC),\\
A_9(\mC)=q^{3m}-1-A_6(\mC)-A_7(\mC)-A_8(\mC).
\end{gather*}
\end{theorem}
\begin{proof}
By the proof of \Cref{thm: qnon_pappus_multilin}, the code $\mC$ satisfies the following properties:
\begin{enumerate}
    \item[(1)] $\dim(\mC)=3m$.
    \item[(2)] For all $U\in\mL(E)_{\geq 3}\setminus\mS$, we have
    $\dim(\mC(U^\perp))=0$.
    \item[(3)] For all $U\in\mS$, we have $\dim(\mC(U^\perp))=m$.
    \item[(4)] For all $U\in\mL(E)_{\leq 2}$, we have
    $\dim(\mC(U^\perp))=(3-\dim(U))m$.
\end{enumerate}
As in the proof of \Cref{thm: qnon_pappus_multilin}, these properties imply that $d=6$, and hence $A_0(\mC)=1$ and $A_1(\mC)=\cdots=A_5(\mC)=0$. Moreover, the same argument gives
\[
    A_6(\mC)=8(q^m-1)
\]
and
\[
    A_7(\C)=(q^m-1)\left(\binom{9}{2}_q-8\binom{3}{1}_q\right).
\]
To determine $A_8(\mC)$, we count codewords according to the number of hyperplanes of $E$ that contain their column space. This allows us to express the total contribution of rank-$6$, rank-$7$, and rank-$8$ codewords to all spaces $\mC(W)$, where $W\in\mL(E)_8$. We double count the pairs $(M,W)$, where $W\subseteq E$ has dimension $8$ and $0\neq M\in\mC(W)$.

First, fix an $8$-dimensional subspace $W\subseteq E$. Then $\dim(W^\perp)=1$, and property (4), applied to $W^\perp$, gives
\begin{equation*}
    \dim(\mC(W))
    =
    \dim(\mC((W^\perp)^\perp))
    =
    (3-\dim(W^\perp))m
    =
    2m.
\end{equation*}
Thus $|\mC(W)|=q^{2m}$, so the number of nonzero elements of $\mC(W)$ is
$q^{2m}-1$. Since the number of $8$-dimensional subspaces of $E$ is
$\binom{9}{8}_q=\binom{9}{1}_q$, counting first by $W$ gives
\begin{equation*}
    \binom{9}{1}_q(q^{2m}-1).
\end{equation*}

We now count the same pairs by fixing the nonzero codeword $M$. Let $r=\rk(M)$. Then $M\in\mC(W)$ if and only if $\colsp_{\F_q}(M)\subseteq W$. Since $\dim(\colsp_{\F_q}(M))=r$, the number of $8$-dimensional subspaces $W\subseteq E$ containing $\colsp_{\F_q}(M)$ is
\begin{equation*}
    \binom{9-r}{8-r}_q.
\end{equation*}
Because $m\geq 9$, every matrix in $\F_q^{9\times m}$ has rank at most $9$, and because $d=6$, the only possible nonzero ranks in $\mC$ are $6,7,8,9$. Rank-$9$ codewords contribute nothing to this count, since their column space is $E$ and hence is not contained in any proper $8$-dimensional subspace. Therefore, counting by the rank of $M$, we obtain
\begin{equation*}
    \binom{9}{1}_q(q^{2m}-1)
    =
    \binom{3}{1}_qA_6(\mC)+\binom{2}{1}_qA_7(\mC)+\binom{1}{1}_qA_8(\mC).
\end{equation*}
Using $\binom{2}{1}_q=q+1$ and $\binom{1}{1}_q=1$, we obtain
\begin{equation*}
    A_8(\mC)=
    \binom{9}{1}_q(q^{2m}-1)-\binom{3}{1}_qA_6(\mC)-(q+1)A_7(\mC).
\end{equation*}

Finally, since $\dim(\mC)=3m$, we have $|\mC|=q^{3m}$. Hence
$$
A_9(\mC)=q^{3m}-A_0(\mC)-A_6(\mC)-A_7(\mC)-A_8(\mC)
=q^{3m}-1-A_6(\mC)-A_7(\mC)-A_8(\mC).
$$
This concludes the proof.
\end{proof}

%%%%%%%%%%%%%%%%%%%%%%%%%%%%%%%%%%%%%%%%%%%%%%%%%%%%%%%%%%%%%%%%%%%%%%%%%%%%%%%%%%%%%%%%%%

\section{Rank-two \emph{q}-matroids on $\F_2^4$}\label{sec: rank2_classification}

In this section, we classify all rank-two $q$-matroids on the ground space $\F_2^4$ with respect to multilinearity for matrix rank-metric codes with $m<4$. Recently, the authors of \cite{IKS25} computed the number of isomorphism classes of rank-two $q$-matroids on $\F_2^4$; previously, such a classification was only known for $\F_2^3$, as characterized in \cite{cj2022}. However, \cite{IKS25} does not provide
explicit representatives or a structural characterization of the different isomorphism classes. We therefore use the characterization from the forthcoming work of the second author
\cite{degenkuehne2026enumqmats}, which is currently in preparation. Throughout this section, we set $E=\F_2^4$.

The authors of \cite{IKS25} show that there are $10$ isomorphism classes of rank-two $q$-matroids on $E$. All $q$-matroids in the same isomorphism class have the same number of bases and are $\F_{2^m}$-representable for the same values of $m$. In
\Cref{table: F_2^4_rank2_qmatroids}, we list these $10$ isomorphism classes according to the number of their bases, and we record all values $m\leq 4$ for which the corresponding class is $\F_{2^m}$-representable. The table summarizes results from
\cite{degenkuehne2026enumqmats}, obtained using the computer algebra system \verb|OSCAR|.

\begin{table}[!h!]
\centering
\ras{1.3}
\begin{tabular}{@{}llll@{}}
\hline
$\mM$ & $|\mB_\mM|$  & \textbf{$\F_{2^m}$-representability} & \textbf{Remark}  \\
\hline
1.  & $16$ & $m\in\{1,2,3,4\}$ & \\
2.  & $24$ & $m\in\{2,3,4\}$ & \\
3.  & $28$ & $m\in\{3,4\}$ & non-paving \\
4.  & $28$ & $m\in\{3,4\}$ & paving \\
5.  & $30$ & $m\in\{2,4\}$ & \\
6.  & $31$ & none & shown in \cite{cj2022,gluesing2022q}\\
7.  & $32$ & $m\in\{3,4\}$ & \\
8.  & $33$ & $m=4$ & \\
9.  & $34$ & $m=4$ & almost uniform \\
10. & $35$ & $m=4$ & uniform \\
\hline
\end{tabular}
\caption{Characterization of the non-isomorphic rank-two $q$-matroids on $\F_2^4$ by the number of their bases and their $\F_{2^m}$-representability for $m\leq 4$.}
\label{table: F_2^4_rank2_qmatroids}
\end{table}

We fix the following notation for the rest of the section.

\begin{notation}
For $i\in[10]$, we denote by $(i,j)$ the isomorphism class in position $i$ of \Cref{table: F_2^4_rank2_qmatroids}, where $j$ is the corresponding number of bases,
that is, $j\in\{16,24,28,30,31,32,33,34,35\}$. For simplicity, we say that the isomorphism class $(i,j)$ is, or is not, representable in one of the senses of \Cref{def:representability}, rather than saying that its elements are, or are not, representable.
\end{notation}

In the remainder of this section, we discuss the possible multilinearity of the $10$ isomorphism classes of rank-two $q$-matroids on $\F_2^4$ with respect to representations by matrix rank-metric codes with $m<4$. By \Cref{def:representability}(3) and
\Cref{rem: multilin_case_m=1}, we only need to consider the cases $1<m<4$. Moreover, for a class to be purely $m$-multilinear, we only need to consider those values of $m$ for which the class is not already $\F_{2^m}$-representable, according to \Cref{table: F_2^4_rank2_qmatroids}. In some of the results below, we choose a representative of an isomorphism class, since all notions of representability in \Cref{def:representability} are invariant under isomorphism. The next result collects the trivial cases arising from \Cref{table: F_2^4_rank2_qmatroids}.

\begin{corollary}\label{coro: not_mul_trivial_cases}
    For all $1<m<4$, the isomorphism classes $(1,16)$ and $(2,24)$ are not purely $m$-multilinear.  
\end{corollary}

\begin{proof}
    By \Cref{table: F_2^4_rank2_qmatroids}, the isomorphism classes $(1,16)$ and $(2,24)$ are $\F_{2^m}$-representable for all $m\in\{2,3\}$. Therefore, for these values of $m$, they cannot be purely $m$-multilinear in the sense of
\Cref{def:representability}(3).
\end{proof}

In the next proposition, we consider the isomorphism classes for which only the case $m=2$ remains to be excluded.

\begin{proposition}\label{prop: not_mul_m=2_cases}
For all $1<m<4$, the isomorphism classes $(3,28)$, $(4,28)$, and $(7,32)$ are not purely $m$-multilinear.
\end{proposition}

\begin{proof}
By \Cref{table: F_2^4_rank2_qmatroids}, all three isomorphism classes are $\F_{2^3}$-representable. Hence, in the range $1<m<4$, it remains only to exclude
the case $m=2$.  Suppose that an $\F_2$-$[4\times 2,k']$ matrix rank-metric code $\mC$ represents a
$q$-matroid in one of these three isomorphism classes. Then $\mC$ must satisfy the following properties, independently of the chosen class:
\begin{enumerate}
    \item The value
    $\rho_\mC(U)=\frac{k'-\dim(\C(U^\perp))}{2}$ is an integer for all
    $U\in\mL(E)$.
    \item The dimension of $\mC$ is $k'=4$.
    \item The code $\mC$ is not right $\F_{4}$-linear.
\end{enumerate}

Property (1) holds because $\rho_\mC$ must be the rank function of a $q$-matroid. Property (2) follows from $\rho_\mC(E)=k'/2=2$, since $\C(E^\perp)=\C(0)=\{0\}$ and we are considering rank-two $q$-matroids. Property (3) is required for a purely $2$-multilinear representation.
    
Using MAGMA, we enumerated all $4$-dimensional subspaces of $\F_2^{4\times 2}$, that is, all $\F_2$-$[4\times 2,4]$ matrix rank-metric codes. We then filtered this list by property (1), and finally checked property (3). After the first two steps, we obtained a finite list of candidate matrix rank-metric codes; however, all of them turned out to be right $\F_{4}$-linear. Hence, no matrix rank-metric code satisfying all three properties exists. Therefore, none of the isomorphism classes $(3,28)$, $(4,28)$, and $(7,32)$ is purely $m$-multilinear for $1<m<4$. 
\end{proof}

\begin{remark}\label{rem: not_mul_uniform_case}
The study of multilinearity for the uniform isomorphism class $(10,35)$ is already covered by \Cref{thm:uniform_not_multilinear}. For the isomorphism class $(8,33)$, it was already shown in \cite[Section~4]{gluesing2022q} that it is not purely $m$-multilinear for $m\leq 3$.
\end{remark}

Next, we study the remaining isomorphism classes, namely $(5,30)$, $(6,31)$ and $(9,34)$, which are more involved than the previous ones. We start with the isomorphism class corresponding to the almost uniform case.

\begin{proposition}\label{prop: not_mul_9_34}
For all $1<m<4$, the isomorphism class $(9,34)$ is not purely $m$-multilinear.
\end{proposition}

\begin{proof}
Let $X$ be a $2$-dimensional subspace of $E$. Since representability in any of the senses of \Cref{def:representability} is invariant under isomorphism, we may choose the representative $\mM=\mA\mU_{2,4}(q,X)$ of the isomorphism class $(9,34)$. We denote its rank function by $\rho_\mM$.

By \Cref{thm: almost_uniform_not_mul}, any matrix rank-metric code representing $\mM$ with $m<4$ must have $m=3$ and parameters $\F_2$-$[4\times 3,6,2]$. Hence, it remains only to exclude the case $m=3$.

Assume, for a contradiction, that $\mM$ is $\F_2^{4\times 3}$-representable. Then there exists an $\F_2$-$[4\times 3,6,2]$ matrix rank-metric code $\mC$ such that, for all $U\in\mL(E)$, we have
\[
    \rho_\mC(U)=\frac{6-\dim(\C(U^\perp))}{3}=\rho_\mM(U).
\]
These equalities imply the following properties:
\begin{enumerate}
    \item[(1)] For all $U\in\mL(E)_{\geq 2}\setminus\{X\}$, we have
    $\dim(\C(U^\perp))=0$.
    \item[(2)] For $U=X$, we have $\dim(\C(X^\perp))=3$.
    \item[(3)] For all $U\in\mL(E)_{\leq 1}$, we have
    $\dim(\C(U^\perp))=3(2-\dim(U))$.
\end{enumerate}
Indeed, property (1) follows from the fact that $\rho_\mM(U)=2$ for all
$U\in\mL(E)_{\geq 2}\setminus\{X\}$. Property (2) follows from
$\rho_\mM(X)=1$, and property (3) follows from $\rho_\mM(U)=\dim(U)$ for all $U\in\mL(E)_{\leq 1}$.

We compute the rank distribution forced by these properties. Since $d=2$, we have $A_0(\mC)=1$ and $A_1(\mC)=0$. By (1) and (2), every rank-$2$ codeword has column space $X^\perp$, and conversely every nonzero element of $\C(X^\perp)$ has rank $2$. Hence, $A_2(\mC)=2^3-1=7$.

We now count rank-$3$ codewords. If $U\in\mL(E)_1$ satisfies $U\subseteq X$, then $X^\perp\subseteq U^\perp$, and by (2) and (3) we have $\C(U^\perp)=\C(X^\perp)$. Thus, such $U$ contribute no rank-$3$ codewords.

There are $\binom{2}{1}_2=3$ one-dimensional subspaces $U\subseteq X$, and therefore $\binom{4}{1}_2-3=12$ one-dimensional subspaces $U$ not contained in $X$. For each such $U$, property (3) gives $\dim(\C(U^\perp))=3$. Moreover, $\C(U^\perp)$ contains no nonzero rank-$2$ codeword, since any rank-$2$ codeword has column space $X^\perp$, which would force $X^\perp\subseteq U^\perp$, equivalently $U\subseteq X$. Hence, all nonzero elements of $\C(U^\perp)$ have rank $3$.

For distinct such subspaces $U,V$, the spaces $\C(U^\perp)$ and $\C(V^\perp)$ intersect trivially. Indeed, a nonzero element in the intersection would have column space contained in $U^\perp\cap V^\perp$, which has dimension $2$, and hence would be a rank-$2$ codeword. Therefore $A_3(\mC)=12(2^3-1)=84$.

Since $\dim(\mC)=6$, we have $|\mC|=2^6=64$. However,
\[
    \sum_{i=0}^3 A_i(\mC)=1+0+7+84=92,
\]
a contradiction. Hence no such matrix rank-metric code exists, and the isomorphism class $(9,34)$ is not purely $m$-multilinear for any $1<m<4$.
\end{proof}

Next, we discuss multilinearity  of the isomorphism class $(5,30)$.

\begin{proposition}\label{prop: not_mul_5_30}
    For all $1<m<4$, the isomorphism class $(5,30)$ is not purely $m$-multilinear.
\end{proposition}

\begin{proof}
All elements in the isomorphism class $(5,30)$ are paving $q$-matroids arising from the construction in \Cref{prop: paving_construction}, where the corresponding set $\mS$ has size $5$. Since representability in any of the senses of \Cref{def:representability} is invariant under isomorphism, we may choose a representative as follows. Let  $\mS\subseteq\mL(E)_2$ be such that $|\mS|=5$ and $\dim(V\cap W)=0$ for all distinct $V,W\in\mS$, and let $\mM_\mS=(\mL(E),\rho_\mS)$ be the corresponding paving $q$-matroid.
    
By \Cref{table: F_2^4_rank2_qmatroids}, we only need to exclude the case $m=3$.
Assume, for a contradiction, that $\mM_\mS$ is $\F_2^{4\times 3}$-representable.
Then there exists an $\F_2$-$[4\times 3,k',d]$ matrix rank-metric code $\mC$ such
that, for all $U\in\mL(E)$,
\[
    \rho_\mC(U)=\frac{k'-\dim(\C(U^\perp))}{3}=\rho_\mS(U).
\]
These equalities imply the following properties:
\begin{enumerate}
    \item[(1)] $\dim(\mC)=k'=6$.
    \item[(2)] For all $U\in\mL(E)_{\geq 2}\setminus\mS$, we have
    $\dim(\C(U^\perp))=0$.
    \item[(3)] For all $U\in\mS$, we have $\dim(\C(U^\perp))=3$.
    \item[(4)] For all $U\in\mL(E)_{\leq 1}$, we have
    $\dim(\C(U^\perp))=3(2-\dim(U))$.
\end{enumerate}

Indeed, (1) follows by evaluating at $U=E$, since $\rho_\mS(E)=2$ and $\C(E^\perp)=\C(0)=\{0\}$. Using $k'=6$, properties (2)--(4) follow directly from the values of $\rho_\mS$.

Property (2), applied to $3$-dimensional subspaces $U\leq E$, shows that $\mC$ contains no nonzero matrix of rank at most $1$. On the other hand, property (3) implies that $\mC$ contains nonzero matrices whose column space is contained in $S^\perp$ for some $S\in\mS$, hence matrices of rank at most~$2$. Therefore the minimum distance of $\mC$ is $d=2$.
    
We now compute the rank distribution forced by these properties. We have
$A_0(\mC)=1$ and $A_1(\mC)=0$. Moreover, by (2) and (3), every rank-$2$ codeword has column space $S^\perp$ for some $S\in\mS$, and conversely every nonzero element of $\C(S^\perp)$ has rank $2$. For distinct $S,S'\in\mS$, the spaces $\C(S^\perp)$ and $\C((S')^\perp)$ intersect trivially, since a nonzero element in their intersection would have rank at most
$\dim(S^\perp\cap (S')^\perp)=\dim((S+S')^\perp)=0$, as $S\cap S'=0$ and $\dim(S+S')=4$. Hence $A_2(\mC)=5(2^3-1)=35$.

We claim that $A_3(\mC)=0$. Since the elements of $\mS$ are pairwise disjoint $2$-dimensional subspaces and $|\mS|=5$, they contain $5\binom{2}{1}_2=15$ one-dimensional subspaces in total. This is the total number $\binom{4}{1}_2=15$ of one-dimensional subspaces of $E$. Hence every one-dimensional subspace $x\leq E$ is contained in a unique member $S\in\mS$. Let $x\in\mL(E)_1$, and let $S\in\mS$ be the unique element such that $x\subseteq S$. Then $S^\perp\subseteq x^\perp$, and hence
$\C(S^\perp)\subseteq \C(x^\perp)$. By (3) and (4), both spaces have dimension $3$, so $\C(x^\perp)=\C(S^\perp)$. Thus, for every $x\in\mL(E)_1$, the space $\C(x^\perp)$ consists only of rank-$2$ codewords and zero. Consequently, $\mC$ contains no rank-$3$ codewords, i.e. $A_3(\mC)=0$.
Finally, since $\dim(\mC)=6$, we have $|\mC|=2^6=64$. However, the rank distribution computed above gives
\[
    \sum_{i=0}^3 A_i(\mC)=1+0+35+0=36,
\]
a contradiction. Therefore, no such matrix rank-metric code exists, and the isomorphism class $(5,30)$ is not purely $m$-multilinear for any $1<m<4$.
\end{proof}

\begin{remark}
    The techniques used in the proofs of Propositions~6.5 and~6.6 can also be used to prove non-multilinearity of the isomorphism class $(8,33)$. Furthermore, this approach is quite different from the one in \cite[Section 4]{gluesing2022q}, where the authors used the computer algebra system SageMath to verify the non-multilinearity computationally.  
\end{remark}

Finally, we investigate the multilinearity property of the isomorphism class $(6,31)$. In this particular case, we consider all possible values $m>1$, in contrast to the previous results, where we restricted to $m<4$. The same result was already proved in \cite[Theorem~4.9]{gluesing2022q}; however, for self-containment, we include a proof adapted to our setting.

\begin{proposition}\label{prop: not_mul_6_31}
For all $m>1$, the isomorphism class $(6,31)$ is not $m$-multilinear.
\end{proposition}

\begin{proof}
All elements in the isomorphism class $(6,31)$ are paving $q$-matroids arising from the construction in Proposition~2.5, where the corresponding set $S$ has size $4$. Since representability in the sense of Definition~2.13 is invariant under isomorphism, we may choose a representative as follows. Let $\mS=\{S_1,S_2,S_3,S_4\}\subseteq \mL(E)_2$, where $E=\F_2^4$, such that $S_i\cap S_j=0$ for all $i\neq j$, and let $\mM_\mS=(\mL(E),\rho_\mS)$ be the corresponding paving $q$-matroid.

Assume for contradiction that $\mM_\mS$ is $\F_2^{4\times m}$-representable for some $m>1$. This means that there exists an
$\F_2$-$[4\times m,k^\prime,d]$ matrix rank-metric code $\mC$ such that, for all $U\in\mL(E)$, we have
\[
    \rho_\mC(U)=\frac{k^\prime-\dim(\mC(U^\perp))}{m}=\rho_\mS(U).
\]
These equalities imply the following properties:
\begin{enumerate}
    \item[(1)] The dimension of $\mC$ is $k^\prime=2m$.
    \item[(2)] For all $U\in\mL(E)_{\geq 2}\setminus \mS$ we have
    $\dim(\mC(U^\perp))=0$.
    \item[(3)] For all $U\in \mS$ we have $\dim(\mC(U^\perp))=m$.
    \item[(4)] For all $U\in\mL(E)_{\leq 1}$ we have
    $\dim(\mC(U^\perp))=m(2-\dim(U))$.
\end{enumerate}
Indeed, the first property follows from $\rho_\mS(E)=2$ and
$\mC(E^\perp)=\la0\ra$. Using $k^\prime=2m$, property (2) follows from
$\rho_\mS(U)=2$ for all $U\in\mL(E)_{\geq 2}\setminus S$, property (3) follows from $\rho_\mS(U)=1$ for all $U\in S$, and property (4) follows from $\rho_\mS(U)=\dim(U)$ for all $U\in\mL(E)_{\leq 1}$.

The four spaces $S_1^\perp,\dots,S_4^\perp$ are again pairwise disjoint
two-dimensional subspaces of $E$. Hence they contain $4(2^2-1)=12$ nonzero vectors. Since $E$ has $2^4-1=15$ nonzero vectors, the remaining three nonzero vectors form a two-dimensional subspace. We denote this space by $S_5^\perp$.
Then $S_1^\perp,\dots,S_5^\perp$ form a $2$-spread of $E$.

We now show that the code conditions force $\dim\mC(S_5^\perp)\geq m$,
contradicting property (2). Since $S_1^\perp\cap S_2^\perp=0$, we have
$\mC(S_1^\perp)\cap\mC(S_2^\perp)=0$. Both spaces have dimension $m$, and
$\dim\mC=2m$, so
\[
    \mC=\mC(S_1^\perp)\oplus\mC(S_2^\perp).
\]

Let
\[
    \pi_1:E=S_1^\perp\oplus S_2^\perp\longrightarrow S_1^\perp,
    \qquad
    \pi_2:E=S_1^\perp\oplus S_2^\perp\longrightarrow S_2^\perp
\]
be the projections with respect to this decomposition. For $j=3,4,5$, the space $S_j^\perp$ is disjoint from both $S_1^\perp$ and $S_2^\perp$. Hence the restriction $\pi_1|_{S_j^\perp}:S_j^\perp\to S_1^\perp$ is an isomorphism. Define
\[
    \alpha_j:=\pi_2\circ(\pi_1|_{S_j^\perp})^{-1}
    :S_1^\perp\longrightarrow S_2^\perp.
\]
Equivalently, for every $x\in S_1^\perp$, the vector
$x+\alpha_j(x)$ is the unique vector in $S_j^\perp$ whose projection onto $S_1^\perp$ is $x$. Thus a vector $v\in E$ lies in $S_j^\perp$ if and only if
\[
    \pi_2(v)=\alpha_j(\pi_1(v)).
\]
By a slight abuse of notation, we also denote by $\alpha_j$ the induced map on matrices whose columns lie in $S_1^\perp$, obtained by applying $\alpha_j$ to each column. Thus, if $M$ has columns $v_1,\dots,v_m$ with $v_\ell\in S_1^\perp$, then $\alpha_j(M):=(\alpha_j(v_1),\dots,\alpha_j(v_m))$.

Since $\mC=\mC(S_1^\perp)\oplus\mC(S_2^\perp)$, every matrix in $\mC$ can be written uniquely as $M_1+M_2$, with $M_1\in\mC(S_1^\perp)$ and
$M_2\in\mC(S_2^\perp)$. For $j=3,4,5$, such a matrix belongs to
$\mC(S_j^\perp)$ if and only if $M_2=\alpha_j(M_1)$. Hence
\[
    \mC(S_j^\perp)=
    \{M_1+\alpha_j(M_1) : M_1\in\mC(S_1^\perp),\
    \alpha_j(M_1)\in\mC(S_2^\perp)\}.
\]

Since $\dim\mC(S_3^\perp)=m=\dim\mC(S_1^\perp)$, the previous description implies
\[
    \alpha_3(\mC(S_1^\perp))=\mC(S_2^\perp).
\]
Similarly, from $\dim\mC(S_4^\perp)=m$, we get
\[
    \alpha_4(\mC(S_1^\perp))=\mC(S_2^\perp).
\]

We claim that $\alpha_5=\alpha_3+\alpha_4$. Indeed, since
$S_3^\perp\cap S_4^\perp=0$, the map
$\alpha_3+\alpha_4:S_1^\perp\to S_2^\perp$ is injective, hence an
isomorphism. The space
\[
    \{x+(\alpha_3+\alpha_4)(x):x\in S_1^\perp\}
\]
is therefore a two-dimensional subspace disjoint from
$S_1^\perp,S_2^\perp,S_3^\perp,S_4^\perp$. Since
$S_1^\perp,\dots,S_5^\perp$ form a spread, this space must be $S_5^\perp$.
Hence $\alpha_5=\alpha_3+\alpha_4$. With the above columnwise convention, this
also gives
\[
    \alpha_5(M)=\alpha_3(M)+\alpha_4(M)
\]
for every matrix $M$ whose columns lie in $S_1^\perp$.

Now let $M_1\in\mC(S_1^\perp)$. Since
$\alpha_3(\mC(S_1^\perp))=\mC(S_2^\perp)$ and
$\alpha_4(\mC(S_1^\perp))=\mC(S_2^\perp)$, we have that 
$\alpha_3(M_1),\alpha_4(M_1)\in\mC(S_2^\perp)$. Since
$\mC(S_2^\perp)$ is an $\F_2$-linear space, also
$\alpha_3(M_1)+\alpha_4(M_1)\in\mC(S_2^\perp)$. Therefore
$\alpha_5(M_1)\in\mC(S_2^\perp)$.

Hence, for every $M_1\in\mC(S_1^\perp)$, the matrix
$M_1+\alpha_5(M_1)$ belongs to $\mC(S_5^\perp)$. Since
$\dim\mC(S_1^\perp)=m$, we obtain
\[
    \dim\mC(S_5^\perp)\geq m.
\]

On the other hand, $S_5^\perp\notin\{S_1^\perp,\dots,S_4^\perp\}$, hence
$S_5\notin \mS$. Therefore, (2) gives $\mC(S_5^\perp)=0$, a contradiction. Thus, no such code $\mC$ exists, and the isomorphism class $(6,31)$ is not $m$-multilinear for any $m>1$.
\end{proof}

%%%%%%%%%%%%%%%%%%%%%%%%%%%%%%%%%%%%%%%%%%%%%%%%%%%%%%%%%%%%%%%%%%%%%%%%%%%%%%%%%%%%%%%%%%
\subsection{Multilinearity classification of all \emph{q}-matroids on $\F_2^3$ and $\F_2^4$}
\label{subsec: classification_F23_F24}

In \Cref{sec: rank2_classification}, we classified all rank-two $q$-matroids on the ground space $\F_2^4$ with respect to pure multilinearity, for matrix rank-metric codes with $m<4$. In this subsection, we extend this classification to all $q$-matroids on $\F_2^4$, again for matrix rank-metric codes with $m<4$. In addition, we provide the corresponding  classification for all $q$-matroids on the ground space $\F_2^3$.

Recall that, by duality, it is enough to consider $q$-matroids of rank $k$ with $0\leq k\leq \lfloor 3/2\rfloor$ on $\F_2^3$, and with
$0\leq k\leq 2$ on $\F_2^4$. Moreover, by
\Cref{ex: uniform_representability} and \Cref{rem: multilin_case_m=1}, we do not need to consider the rank-zero $q$-matroids $\mU_{0,3}(2)$ and $\mU_{0,4}(2)$. Hence, after the discussion in
\Cref{sec: rank2_classification}, it remains only to treat rank-one $q$-matroids.

The following result characterizes pure multilinearity for rank-one $q$-matroids on $E=\F_q^n$.

\begin{theorem}\label{thm: rank1_non_mul}
   Let $\mM=(\mL(E),\rho_\mM)$ be a $q$-matroid of rank one on $E=\F_q^n$. Then there exists no $m>1$ such that $\mM$ is purely $m$-multilinear.
\end{theorem}

\begin{proof}
    By \cite[Remark~4.8]{DK25}, a rank-one $q$-matroid is determined, up to isomorphism, by its loop space. If the loop space has dimension $t$, where $0\leq t\leq n-1$, then the corresponding rank-one $q$-matroid is $\F_{q^{n-t}}$-representable. 

    We may assume, without loss of generality, that the loop space is $L=\la e_1,\ldots,e_t\ra$, where $e_i$ denotes the $i$-th standard basis vector of $\F_q^n$. The case $t=0$ is the uniform $q$-matroid $\mU_{1,n}(q)$, which is already covered by \Cref{thm:uniform_not_multilinear}. The case $t=n-1$ is $\F_q$-representable, and therefore cannot be purely $m$-multilinear. Hence, we may assume $0<t<n-1$.
    
    By \cite[Remark~4.8]{DK25}, the $q$-matroid $\mM$ is represented by an $\F_{q^{n-t}}$-linear rank-metric code with generator matrix \[
    G=
    \begin{pmatrix}
        0&\ldots&0&1&\alpha&\alpha^2&\ldots&\alpha^{n-t-1}
    \end{pmatrix},
    \]
where $\alpha$ is a primitive element of $\F_{q^{n-t}}$ over $\F_q$.
Moreover, for every $m\geq n-t$, the same construction gives an $\F_{q^m}$-linear representation of $\mM$. Thus, to exclude pure
$m$-multilinearity, it remains to rule out matrix-code representations in $\F_q^{n\times m}$ for $1<m<n-t$.

    Assume, for a contradiction, that there exists an $\F_q$-$[n\times m,k^\prime,d]$ matrix rank-metric code $\mC\leq\F_q^{n\times m}$, with $1<m<n-t$, such that, for all $U\in\mL(E)$,
\[
    \rho_\mC(U)
    =
    \frac{k^\prime-\dim(\mC(U^\perp))}{m}
    =
    \rho_\mM(U).
\]
These equalities imply the following properties:
\begin{enumerate}
    \item[(1)] $\dim(\mC)=k^\prime=m$.
    \item[(2)] For all $U\in\mL(E)\setminus\mL(L)$, we have
    $\dim(\mC(U^\perp))=0$.
    \item[(3)] For all $U\in\mL(L)$, we have
    $\dim(\mC(U^\perp))=m$.
\end{enumerate}
Indeed, property (1) follows by evaluating at $U=E$, since $\rho_\mM(E)=1$ and $\mC(E^\perp)=\mC(0)=\la0\ra$. Using $k^\prime=m$,
property (2) follows from $\rho_\mM(U)=1$ for
$U\notin\mL(L)$, while property (3) follows from $\rho_\mM(U)=0$ for $U\in\mL(L)$.

We now show that these conditions force the minimum rank distance of $\mC$ to be $n-t$. By property (2), applied to all subspaces
$U\in\mL(E)_{t+1}$ with $U\not\subseteq L$, the code $\mC$ contains no nonzero matrix of rank at most $n-t-1$. Indeed, if $M\in\mC$ were nonzero with $\operatorname{rk}(M)\leq n-t-1$, then $\operatorname{colsp}(M)$ would be contained in some subspace $W\leq E$ of dimension $n-t-1$. Hence $W=U^\perp$ for some $U\in\mL(E)_{t+1}$. Since $\dim(U)=t+1$, we have $U\not\subseteq L$, and thus property (2) gives $\mC(U^\perp)=0$, a contradiction.

On the other hand, property (3), applied to $U=L$, gives $\dim(\mC(L^\perp))=m$. Hence $\mC$ contains a nonzero matrix whose column space is contained in $L^\perp$, and therefore whose rank is at most $\dim(L^\perp)=n-t$. Since ranks at most $n-t-1$ are excluded, the minimum rank distance of $\mC$ is $d=n-t$.
But every matrix in $\F_q^{n\times m}$ has rank at most $m$. Therefore $d\leq m$, and so $n-t\leq m$, contradicting the assumption $m<n-t$.
Thus no such matrix rank-metric code exists for $1<m<n-t$.

For $m\geq n-t$, the $q$-matroid $\mM$ is already $\F_{q^m}$-representable, and hence it is not purely $m$-multilinear. Therefore, $\mM$ is not purely $m$-multilinear for any $m>1$. 
\end{proof}

The following classification result is a direct consequence of
\Cref{thm:uniform_not_multilinear,thm: rank1_non_mul} and Propositions~\ref{prop: not_mul_m=2_cases}, \ref{prop: not_mul_9_34}, \ref{prop: not_mul_5_30}, and \ref{prop: not_mul_6_31}.

\begin{theorem}\label{thm: mul_class_F23_F24}
Let $E=\F_2^n$, with $n\in\{3,4\}$. Then the following hold:
\begin{enumerate}
    \item[(1)] For every $m>1$, there is no $q$-matroid on $E=\F_2^3$ that is purely $m$-multilinear.
    \item[(2)] For every $1<m<4$, there is no $q$-matroid on $E=\F_2^4$ that is purely $m$-multilinear.
\end{enumerate}
\end{theorem}

\begin{remark}
The theorem shows that for $n=3$ there are no purely multilinear cases at all. For $n=4$, any possible purely multilinear example would have to come from a matrix rank-metric code in $\F_2^{4\times m}$ with $m\geq 4$. However, even the case $m=4$ is already computationally demanding, making an exhaustive computer search difficult. Thus, finding such an example, if it exists, will likely require a more theoretical construction.
\end{remark}
% \textcolor{red}{Heide and Ben already showed in their paper that the 33-bases one is not multilinearly representable and also the 31-bases}

%%%%%%%%%%%%%%%%%%%%%%%%%%%%%%%%%%%%%%%%%%%%%%%%%%%%%%%%%%%%%%%%%%%%%%%%%%%%%%%%%%%%%%%%%%

\section{Conclusion and open questions}\label{sec: conclusion}

In this paper, we introduced a $q$-analogue of multilinear
representability for $q$-matroids using matrix rank-metric codes. We also gave an equivalent description of the $q$-rank function of a representable $q$-polymatroid in terms of a generator tensor. Using this description, we showed that representability descends from a representable $q$-polymatroid to its projectivization polymatroid.
We then investigated pure multilinearity for several classes of $q$-matroids, including uniform $q$-matroids, almost uniform $q$-matroids, rank-one $q$-matroids, the non-Pappus $q$-matroid, and $q$-matroids on $\F_2^3$ and $\F_2^4$. In all cases considered, we found strong obstructions to pure $m$-multilinearity. More precisely, most of these $q$-matroids do not admit a purely $m$-multilinear representation for values of $m$ strictly smaller than the dimension of the ground space. For uniform $q$-matroids, rank-one $q$-matroids, $q$-matroids on $\F_2^3$, and some $q$-matroids on $\F_2^4$, this nonexistence extends to all $m>1$.
In the remaining cases, we determined necessary parameters for any possible purely $m$-multilinear representation. However, even in the smallest unresolved cases, an exhaustive computer search appears to be computationally difficult. Thus, new theoretical constructions or obstructions seem necessary.

We end with some open questions and possible directions for future research.

\begin{enumerate}
    \item[(a)] Does there exist a $q$-matroid that is purely $m$-multilinear for some $m>1$? Equivalently, does there exist an $m$-multilinear matrix rank-metric code whose associated $q$-matroid is not $\F_{q^m}$-representable?

    \item[(b)] Is the non-Pappus $q$-matroid purely $m$-multilinear for some
    $m\geq 9$?

    \item[(c)] Are almost uniform $q$-matroids ever purely multilinear? More
    precisely, for which parameters can an almost uniform $q$-matroid
    $\mA\mU_{k,n}(q,X)$ admit a matrix rank-metric representation with
    $m=n-1$?
\end{enumerate}

\bigskip

\section*{Acknowledgments}
%%%%%%%%%%%%%%%%%%%%%%%%%%%%%%%%%%%%%%%%%%%%%%%%%%%%%%%%%%%%%%%%%%%%%%%%%%%%%%%%%%%%%%%%%%%%%%%%%
The authors are thankful to Lukas K\"uhne for fruitful discussions. The first author is supported by the grant ANR-24-CPJ1-0075-01. The second author is supported by the Deutsche Forschungsgemeinschaft (DFG, German Research Foundation) through grants SFB-TRR 358/1 2023 -- 491392403 and SPP 2458 -- 539866293. The second author would especially like to thank Lukas K\"uhne for his supervision and guidance throughout the research process.

\bigskip

\bibliographystyle{abbrv}
\bibliography{references}
\end{document}